\documentclass[11pt]{amsart}

\usepackage{MnSymbol}
\usepackage{bbm, amsbsy}
\usepackage[dvips]{graphics}
\usepackage{graphicx}
\usepackage{graphics}
\usepackage{psfrag}
\newtheorem{theorem}{Theorem}[section]
\newtheorem{lemma}[theorem]{Lemma}

\newtheorem{cor}[theorem]{Corollary}
\newtheorem{corollary}[theorem]{Corollary}

\theoremstyle{definition}

\theoremstyle{remark}
\newtheorem{remark}[theorem]{Remark}

\newcommand{\MM}{\mathbb{M}}
\newcommand{\RR}{\mathbb{R}}
\newcommand{\dd}{\mathrm{d}}

\usepackage{comment}
\usepackage{tikz}
\usepackage{float}
\usepackage{relsize}
\usepackage{enumitem}
\usepackage[T1]{fontenc}
\usepackage[english]{babel}
\usepackage{color}
\usepackage{cite}
\usepackage{graphicx}
\usepackage{caption}
\usepackage{subcaption}
\usepackage{cite}

\date{\today}

\begin{document}

\sloppy
\title{Intersection probabilities for flats in hyperbolic space}

\author{Ercan S\"onmez}
\address{Ercan S\"onmez, Ruhr University Bochum, Faculty of Mathematics, Bochum, Germany}
\email{ercan.soenmez\@@{}rub.de}

\author{Panagiotis Spanos}
\address{Panagiotis Spanos, Ruhr University Bochum, Faculty of Mathematics, Bochum, Germany}
\email{panagiotis.spanos\@@{}rub.de} 

\author{Christoph Th\"ale}
\address{Christoph Th\"ale, Ruhr University Bochum, Faculty of Mathematics, Bochum, Germany}
\email{christoph.thaele\@@{}rub.de} 

\date{}

\begin{abstract}
Consider the $d$-dimensional hyperbolic space $\MM_K^d$ of constant curvature $K<0$ and fix a point $o$ playing the role of an origin. Let $\mathbf{L}$ be a uniform random $q$-dimensional totally geodesic submanifold (called $q$-flat) in $\MM_K^d$ passing through $o$ and, independently of $\mathbf{L}$, let $\mathbf{E}$ be a random $(d-q+\gamma)$-flat in $\MM_K^d$ which is uniformly distributed in the set of all $(d-q+\gamma)$-flats intersecting a hyperbolic ball of radius $u>0$ around $o$. We are interested in the distribution of the random $\gamma$-flat arising as the intersection of $\mathbf{E}$ with $\mathbf{L}$. In contrast to the Euclidean case, the intersection $\mathbf{E}\cap \mathbf{L}$ can be empty with strictly positive probability. We determine this probability and the full distribution of $\mathbf{E}\cap \mathbf{L}$. Thereby, we elucidate crucial differences to the Euclidean case. Moreover, we study the limiting behaviour as $d\uparrow\infty$ and also $K\uparrow 0$. Thereby we obtain a phase transition with three different phases which we completely characterize, including a critical phase with distinctive behavior and a phase recovering the Euclidean results. In the background are methods from hyperbolic integral geometry.
\end{abstract}

\thanks{This work has been supported by the German Research Foundation (DFG) via SPP 2265 Random Geometric Systems}
\maketitle

\baselineskip=18pt

\section{Introduction}

The computation and analysis of intersection or hitting probabilities is a classical theme of geometric probability in which integral-geometric tools play a significant role. To give an example, let $K_0$ and $M$ be two convex bodies in $\RR^d$, and let $X_{K_0,M}$ be a random congruent copy of $M$ meeting $K_0$. Then, for another convex body $N\subset K_0$, the intersection probability $\mathbb{P}(X_{M,K_0}\cap N\neq\varnothing)$ can be determined by means of the principal kinematic formula from integral geometry. This yields an expression in terms of the intrinsic volumes of $K_0$, $M$ and $N$, see \cite[Theorem 8.4.1]{sw} and also \cite[Chapter III.14.2]{s}.  

Other classical intersection probabilities are concerned with the intersection of random linear or random affine subspaces in $\RR^d$. They can be determined by means of appropriate integral-geometric transformation formulas of Blaschke-Petkantschin type, see Chapters 7.2 and 8.4 in \cite{sw} and also \cite{s}. Closely connected to these problems, the following set-up has only recently been considered in \cite{dkt}. Let $\mathbf{L}$ be a random linear subspace in $\RR^d$ of dimension $q\in\{1,\ldots,d-1\}$ and $\mathbf{E}$ be a random affine subspace in $\RR^d$ of dimension $d-q+\gamma$ for some $\gamma\in\{0,\ldots,q-1\}$ which has non-empty intersection with a ball of a given radius $u>0$, and is independent of $\mathbf{L}$. What is the distribution of $\mathbf{E}\cap \mathbf{L}$? To answer this question a new integral-geometric transformation formula had been developed. It mixes the purely linear with the purely affine case, which were previously known in the literature. 

The present paper can be seen as a natural continuation of \cite{dkt} in which we extend the results to $d$-dimensional hyperbolic spaces of constant negative curvature $K<0$. Such an extension is interesting, as new phenomena appear for such negatively curved spaces, which cannot be observed in the Euclidean set-up. For example, it is not necessarily the case that the natural hyperbolic analogues of the linear and affine subspace $\mathbf{L}$ and $\mathbf{E}$, as considered above, do intersect. In fact, the intersection can be empty with strictly positive probability and one of our goals is to compute this probability explicitly in terms of the underlying dimension parameters $d$, $q$, $\gamma$ and the curvature $K$. To better illustrate the different intersection scenarios in hyperbolic space, we present the problem within the Poincaré disk model, see Figure \ref{fig:intersection_configurations}. In this model, the hyperbolic analogues of affine subspaces appear as circular arcs that intersect the boundary of the disk at right angles, or as diameters of the disk. The figure contains four subfigures that visually represent the possible configurations. Moreover, we fully determine the distribution of the intersection $\mathbf{E} \cap \mathbf{L}$ and investigate moment properties of the distance of $\mathbf{E} \cap \mathbf{L}$ to a fixed reference point through which $\mathbf{L}$ passes. Thereby we point out crucial differences to the properties in the Euclidean space. A key result of our study is the identification of a phase transition for the intersection probabilities in hyperbolic spaces. Specifically, we demonstrate three distinct phases that highlight the significant roles played by both curvature and dimension, as characterized in Theorem \ref{phase transition} below. This phase transition reveals an intricate interplay between curvature and dimension in determining the probability of intersections in hyperbolic spaces. By varying these parameters, we elucidate how the geometric properties of the space influence the likelihood of intersection. Our results reveal three phases, emphasizing crucial differences from the Euclidean case. Each phase is concerned with different dynamics between curvature and dimension. In the first phase, where the curvature tends to zero faster than the dimension increases, the intersection probability tends to 1, recovering the Euclidean results. In the second phase, where the dimension tends to infinity faster than the curvature tends to zero, the intersection probability tends to 0. The third, critical phase exhibits distinctive behaviour, where the interplay between curvature and dimension results in a probability that converges to a non-trivial constant that we make explicit. %We fully characterize all these phases, elucidating the intricate differences and transitions between hyperbolic and Euclidean geometries.

The study of the asymptotic behavior of random geometric systems in high dimensions is a popular topic in the fields of stochastic and convex geometry. Often in such situations a phase transition in the behavior of geometric parameters can be observed. In the last years, several results in this direction have been studied and remain a significant focus, see \cite{BonChaGro,DyerFureMcDi,FriePegTko,ChaTkoVrit,BrazGianPafi,HugSchneider22,Pivovarov2007}. Our paper can be regarded as a contribution and an extension to this branch of research, where we study for the first time high-dimensions (and small curvatures) phase transitions in a hyperbolic set-up. The extension from the Euclidean to the hyperbolic set-up also lines up well with another recent trend in stochastic geometry in which the study of random polytopes, random graphs or random tessellations has been extended from Euclidean to hyperbolic spaces as well, see \cite{BenjaminiEtAlAlea,BesauThaele,FountoulakisYukich,GodlandKabluchkoThaeleBetaStar,HansenMueller,HeroldHugThaele,Isokawa2d,Isokawa3d,TykessonCalka} for example.

\begin{figure}[t!]
  \centering
  \begin{subfigure}[b]{0.24\textwidth}
    \includegraphics[width=\textwidth]{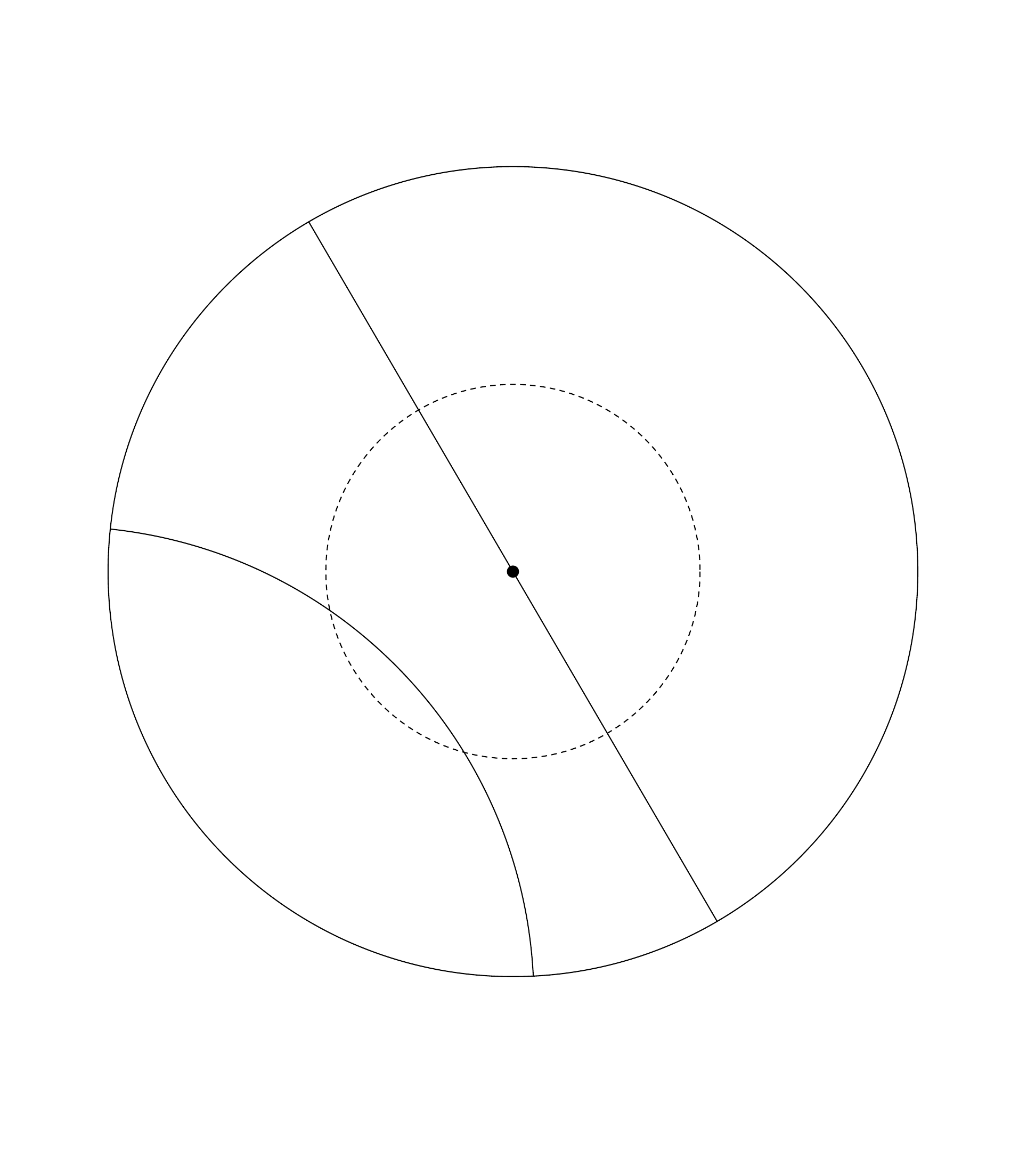}
    %\caption{}
    \label{fig:no_intersection}
  \end{subfigure}
  \hfill
  \begin{subfigure}[b]{0.24\textwidth}
    \includegraphics[width=\textwidth]{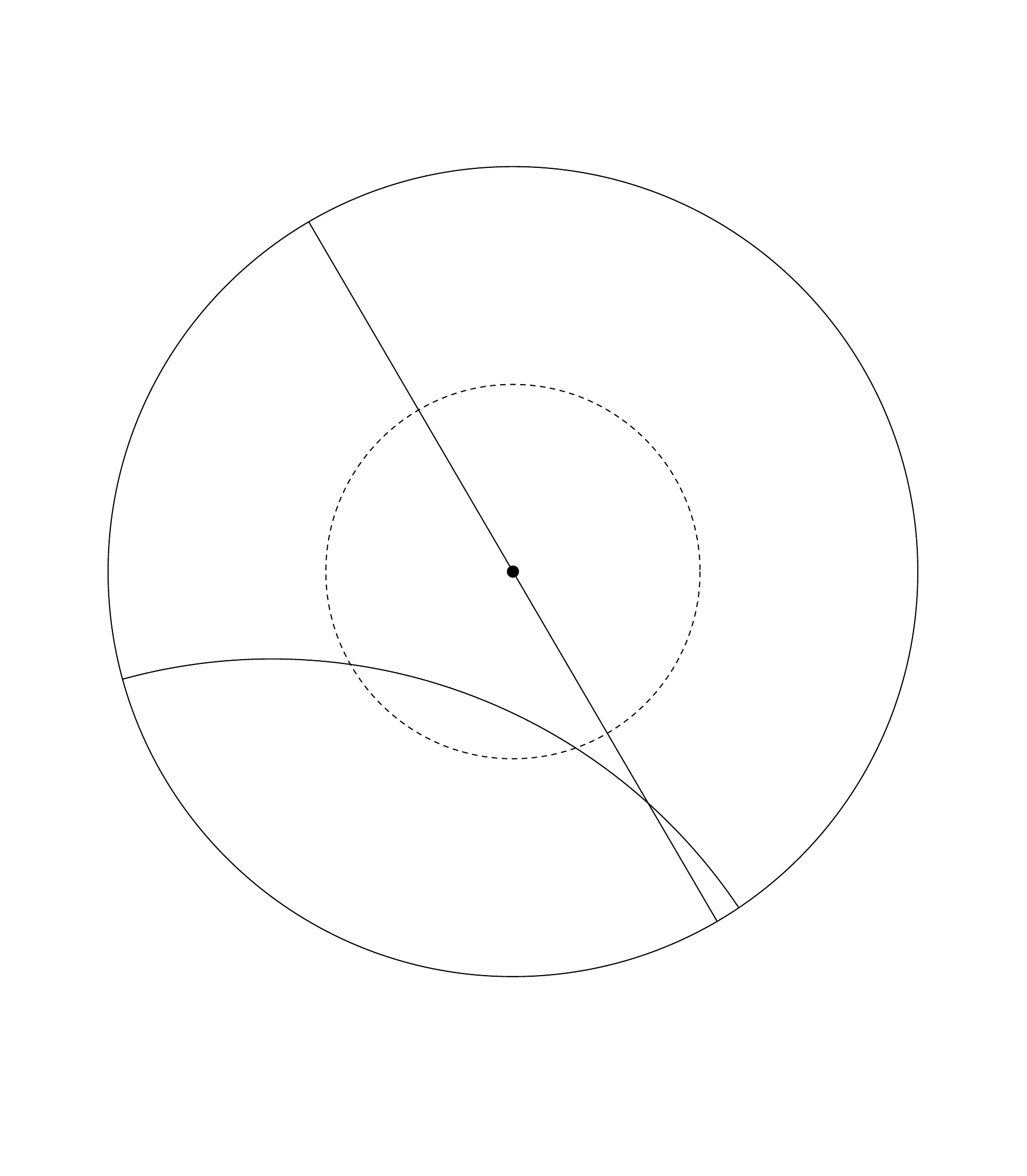}
    %\caption{Intersect Outside}
    \label{fig:intersect_outside}
  \end{subfigure}
  \hfill
  \begin{subfigure}[b]{0.24\textwidth}
    \includegraphics[width=\textwidth]{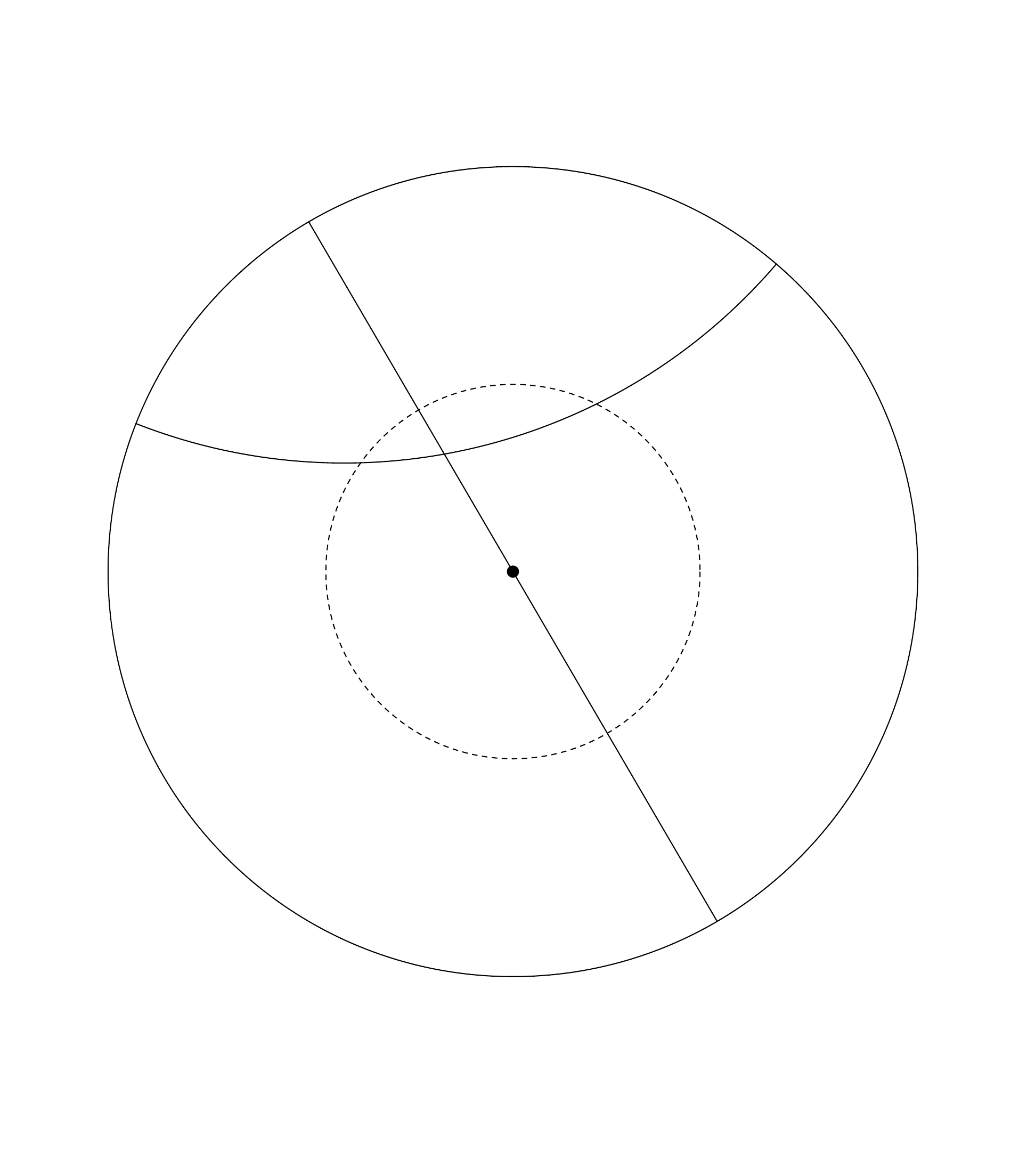}
    %\caption{Intersect Inside}
    \label{fig:intersect_inside}
  \end{subfigure}
  \hfill
  \begin{subfigure}[b]{0.24\textwidth}
    \includegraphics[width=\textwidth]{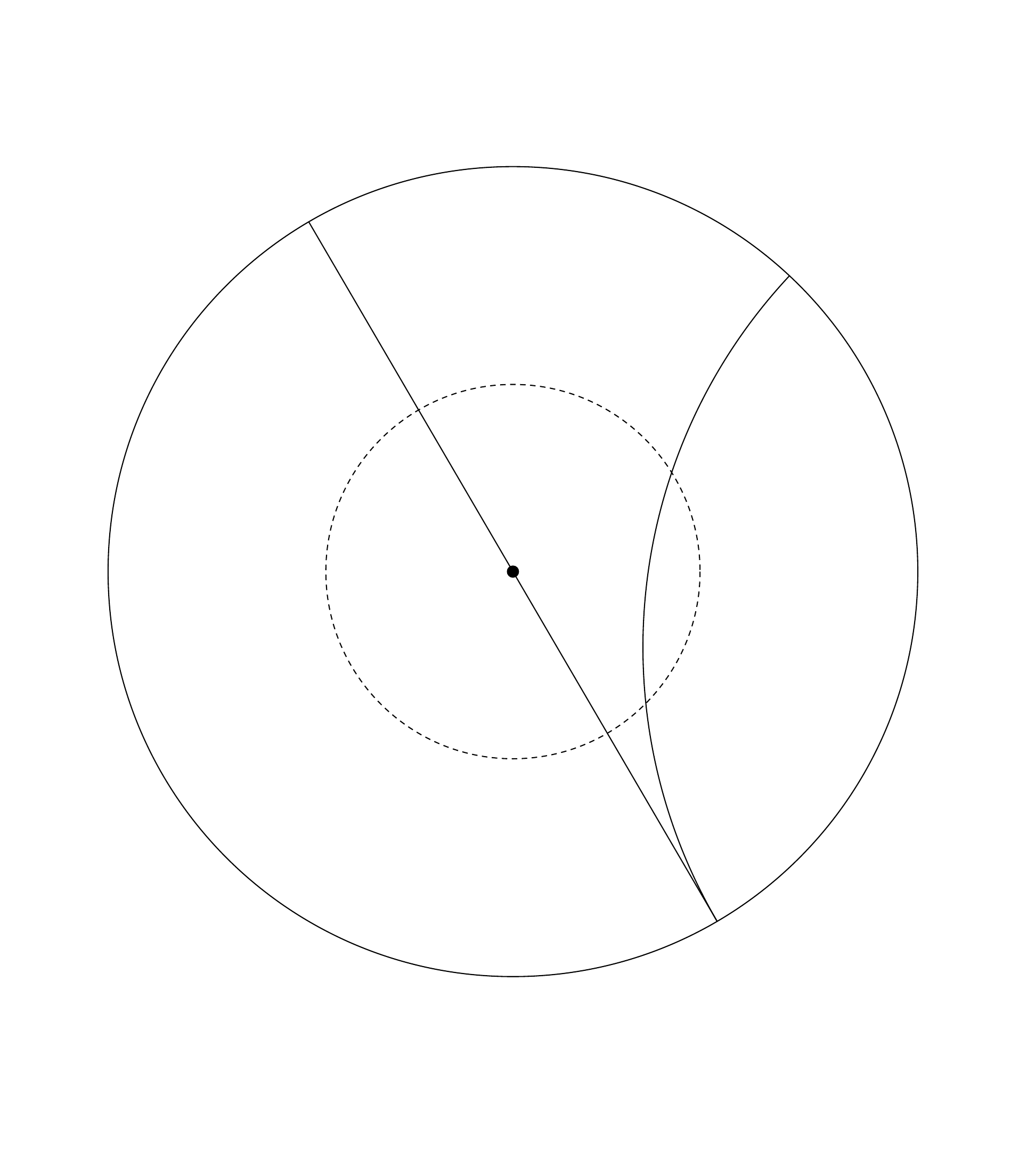}
    %\caption{asymptotic no intersection}
    \label{fig:asymptotic_no_intersection}
  \end{subfigure}
  \caption{Intersection configurations in hyperbolic space. Each picture shows the ball of radius 1 centred at the `origin' $o$ in the Poincaré disk model, a subspace crossing $o$ and a flat intersecting the ball. From left to right the following scenarios are depicted: no intersection, intersection outside the ball, intersection inside the ball, no intersection but the flats share a common ideal point on the hyperbolic boundary.}
  \label{fig:intersection_configurations}
\end{figure}

\medspace

The remaining parts of this paper are structured as follows. In Section \ref{sec:Prelim} we introduce some notation and concepts from hyperbolic geometry. Our main results are formulated in Section \ref{sec:MainResults}. In Section \ref{sec:SpaceK<0} focus is on the space of constant curvature $K<0$, where we develop some integral-geometric formulas which will allow us to use the Blaschke-Petkantschin type from \cite{bht} also in spaces of arbitrary negative curvature parameters. The remaining sections are devoted to the proofs of the main results in Section \ref{sec:MainResults}.

\section{Preliminaries}\label{sec:Prelim}

Before we present the main results of this paper in the next section, we introduce some notation. Fix $d \geq 2$ and $K\leq 0$. By $\MM_K^d$ we indicate the $d$-dimensional standard space of constant sectional curvature $K$. More precisely, if $K=0$, then $\MM_0^d$ can be identified with the $d$-dimensional Euclidean space $\RR^d$, whereas for $K=-1$, $\MM_{-1}^d$ is the hyperbolic space $\mathbb{H}^d$ of constant sectional curvature $-1$. The Riemannian metric on $\MM_K^d$ is denoted by $d_K(\,\cdot\,,\,\cdot\,)$. Throughout this paper we fix a point $o\in\MM_K^d$, which will be referred to as the origin. If $K=0$ we choose for $o$ in fact the point with coordinates $(0,\ldots,0)$. For $u>0$ we write $uB_K^d:=\{x\in\MM_{K}^d:d_K(x,o)\leq u\}$ for the $d$-dimensional closed ball in $\MM_K^d$ of radius $u$ centred at $o$. Further, for a set $F \subset \MM_K^d$ we define $d_K(o,F):=\inf\{d_K(o,x):x\in F\}$ as the distance of $F$ to the origin. 

For $q \in \{0, \ldots , d-1\}$ we write $G_K(d,q)$ for the Grassmannian of $q$-dimensional totally geodesic submanifolds of $\MM_K^d$ that pass through $o$. Moreover, we let $A_K(d,q)$ be the larger space of all $q$-dimensional totally geodesic submanifolds of $\MM_K^d$. We refer to the elements of $A_K(d,q)$ (and consequently also of $G_K(d,q)$) as $q$-flats in what follows. If $K=0$ then elements in $G_0(d,q)$ are the $q$-dimensional linear subspaces and the elements of $A_0(d,q)$ the $q$-dimensional affine subspaces of $\RR^d$. For $K<0$ one can use, for example, the Beltrami-Klein model for $\MM_K^d$ and choose for $o$ the centre of the ball in which the model is realized, see \cite[Chapter 6]{r} and Section \ref{sec:SpaceK<0} below, where the model is described in detail. In this model, elements in $A_K(d,q)$ are represented as the non-empty intersections of $q$-dimensional affine subspaces of $\RR^d$ with the $d$-dimensional open Euclidean ball of radius $1/\sqrt{-K}$. Similarly, the elements of $G_K(d,q)$ correspond to intersections of the $d$-dimensional open Euclidean ball of radius $1/\sqrt{-K}$ with linear subspaces of $\RR^d$ passing through the centre of that ball.

It is well known that the space $G_K(d,q)$ carries a unique Haar probability measure $\nu_{q,K}$ which is invariant under all isometries of $\MM_K^d$ that fix the origin. Furthermore, by Equation (17.41) in \cite{s} the space $A_K(d,q)$ can be supplied with a measure $\mu_{q,K}$ which is invariant under the full isometry group of $\MM_K^d$ and is uniquely determined up to a multiplicative constant. If $K=0$ we choose
\begin{align}\label{eq:MuqEuclidean}
\mu_{q,0}(\,\cdot\,) = \int_{G_0(d,q)}\int_{L^\perp}{\bf 1}{\{L+x\in\,\cdot\,\}}\,\lambda_{L^\perp}(dx)\nu_{q,0}(dL),
\end{align}
where $\lambda_{L^\perp}$ is the Lebesgue measure on $L^\perp$, the orthogonal complement of $L$. On the other hand, if $K<0$ we take
\begin{equation}\label{eq:MuqK}
\mu_{q,K}(\,\cdot\,)=\int_{G_{K}(d,d-q)}\int_{L}\mathbf{1}\{ E(L,x)  \in\, \cdot\,\} \,\cosh^{q}(\sqrt{-K}\, d_K(x,o))\, \mathcal{H}_{K}^{d-q}(\dd x)\nu_{d-q,K}(\dd L).
\end{equation}
Here, $E(L,x)$ is the unique $q$-flat of $\MM_{K}^d$ passing through $x$ orthogonally to $L$ and $\mathcal{H}_K^{d-q}$ stands for the $(d-q)$-dimensional Hausdorff measure with respect to the hyperbolic structure of curvature $K$.

Next, for $u>0$ we let
$$
[uB^d_{K}]_{q,K}:=\{ E \in A_K(d,q) : E \cap uB^d_{K}  \neq \varnothing \}
$$
be the set of  $q$-flats in $\MM_K^d$ that hit the ball of radius $u$ around $o$. The invariant measure $\mu_{q,K}([uB^d_{K}]_{q,K})$ of the set $[uB^d_{K}]_{q,K}$ can be determined by Crofton's formula, see \cite[Chapter 5.1]{sw} for $K=0$ and \cite[Chapter IV.17.4.]{s} for $K<0$. To formulate the result, for $n \in \mathbb{N}$ we define
$$
\omega_n := \frac{2\pi^{n/2}}{\Gamma (\frac{n}{2})}
$$
as the surface content of the $n$-dimensional Euclidean unit sphere. Then it holds that
\begin{align}\label{normconstant}
\nonumber C_K(d,q,u) &:= \mu_{q,K}([uB^d_{K}]_{q,K})\\
\nonumber & = \int_{A_K(d,q)}{\bf 1}{\{E\cap uB_K^d\neq\varnothing\}}\,\mu_{q,K}(\dd E) \\
\nonumber &= \begin{cases}
\omega_{d-q}\int_0^u r^{d-q-1}\,\dd r &: K=0\\
\omega_{d-q} (\sqrt{-K})^{-d+q+1} \int_0^u \cosh^{q}(\sqrt{-K}r)\sinh^{d-q-1}(\sqrt{-K}r) \,\dd r &: K < 0 
\end{cases} \\
&= \begin{cases}
{1\over d-q}\omega_{d-q}u^{d-q} &: K=0\\
\omega_{d-q} (\sqrt{-K})^{-d+q+1} \int_0^u \cosh^{q}(\sqrt{-K}r)\sinh^{d-q-1}(\sqrt{-K}r) \,\dd r &: K < 0  .
\end{cases}
\end{align}

\section{Main results}\label{sec:MainResults}

In this section we will present the main results of this paper. There is a number of crucial differences to the Euclidean case \cite{dkt}, as we already indicated in the introduction. We will highlight them separately below. We start by formalizing the set-up under consideration. For that purpose, we fix dimension parameters $d\geq 2$, $q\in\{1,\ldots,d-1\}$ and $\gamma\in\{0,\ldots,q-1\}$, and a curvature parameter $K\leq 0$. Suppose that
\begin{itemize}
\item $\mathbf{L}$ is a random element of $G_{K}(d,q)$ with distribution $\nu_{q,K}$ 
\end{itemize}
and
\begin{itemize}
\item $\mathbf{E}$ is a random element of $[uB^d_{K}]_{d-q+\gamma,K}$ for some $u>0$ whose distribution is the restriction of $C_{K}(d-q+\gamma,u)^{-1} \mu_{d-q+\gamma,K}$ to $[uB^d_{K}]_{d-q+\gamma,K}$. 
\end{itemize}
Moreover, we assume that $\mathbf{L}$ and $\mathbf{E}$ are stochastically independent and note that restriction of $C_{K}(d-q+\gamma,u)^{-1} \mu_{d-q+\gamma,K}$ to $[uB^d_{K}]_{d-q+\gamma,K}$ is indeed a probability measure according to \eqref{normconstant}. In the following we denote
$$
R_K(x) = R(x) := \frac{\sinh (\sqrt{-K} x)}{\sqrt{-K (1+\sinh^2 (\sqrt{-K} x))}},\qquad x>0.
$$
In sharp contrast to the Euclidean case $K=0$ considered in \cite{dkt} the intersection of $\mathbf{L}$ with $\mathbf{E}$ can be empty with strictly positive probability if $K<0$, see Figure \ref{fig:intersection_configurations}. Our first goal is to determine the probability that $\mathbf{L}$ and $\mathbf{E}$ intersect non-trivially. In what follows we denote the minimum of two real numbers $x,y\in\RR$ by $x\wedge y$.

\begin{theorem} \label{intersection probability}
	For $K<0$ we have that
	$$\mathbb{P} ( \mathbf{E} \cap \mathbf{L} \neq \varnothing ) = \frac{D(d,q,\gamma)\omega_{d-\gamma}}{C_{K}(d,d-q+\gamma,u)}
\int_{0}^{1/\sqrt{-K}}   {r^{q-\gamma-1}}
\int_{0}^{1\wedge \frac{R(u)}{r}} 
\frac{z^q(1-z^2)^{\frac{d-q}{2}-1}}{(1+Kr^2z^2)^{\frac{d+1}{2}} }\, \dd z
\dd r,$$
where $C_{K}(d,d-q+\gamma,u)$ is given by \eqref{normconstant} and $D(d,q,\gamma)$ is defined as
$$D(d,q,\gamma) := \frac{\omega_{\gamma +1} \omega_{q-\gamma} \omega_{d-q}}{\omega_{d-q+\gamma+1} \omega_{d-\gamma}}. $$ 
\end{theorem}

As a consequence, we can determine the limit in high dimensions, that is, as $d\to\infty$, of the probability that $\mathbf{E}$ and $\mathbf{L}$ intersect non-trivially. 

\begin{corollary}\label{cor:IntersectionProbabdtoInfinity}
For fixed $K<0$ it holds that
$$
\lim_{d\to \infty}\mathbb{P} ( \mathbf{E} \cap \mathbf{L}\neq \varnothing ) = 0.
$$
\end{corollary}

\begin{figure}[t]
	\centering
	\includegraphics[width=0.6\textwidth]{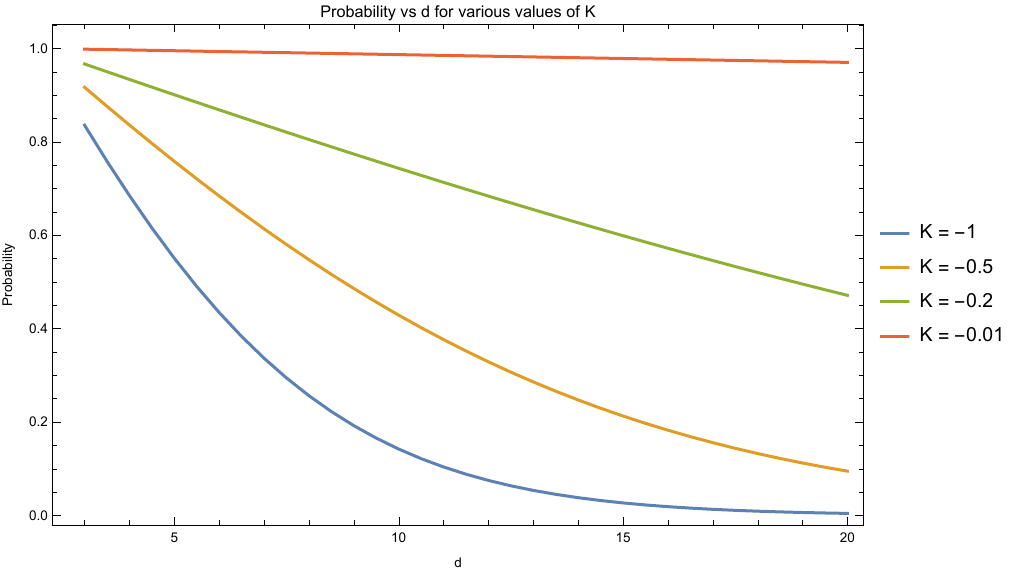}
	\caption{The probability $\mathbb{P} ( \mathbf{E} \cap \mathbf{L} \neq \varnothing )$ plotted against the dimension $d$ for different choices of curvatures $K$.}
	\label{fig:dimensionplot}
\end{figure}

Figure \ref{fig:dimensionplot} illustrates the findings of Corollary \ref{cor:IntersectionProbabdtoInfinity}. A notable difference to the Euclidean setting treated in \cite{dkt} is that for spaces of strictly negative curvature we have $\lim_{d\to \infty}\mathbb{P} ( \mathbf{E} \cap \mathbf{L} \neq \varnothing ) = 0$, whereas in the Euclidean space it holds $\mathbb{P} ( \mathbf{E} \cap \mathbf{L} \neq \varnothing ) = 1$ for every $d \in \mathbb{N}$.

Our next goal is to determine the distribution of $\mathbf{E}\cap \mathbf{L}$. By isometry invariance of the distributions of $\mathbf{L}$ and $\mathbf{E}$, all information about this distribution is contained in the random variable $d_K(o, \mathbf{E} \cap \mathbf{L} )$, which describes the distance of the intersection to the origin and takes values in the extended positive real half-axis $[0,+\infty]$. The following result describes the distribution of this random variable.

\begin{theorem}\label{density}
Let $K<0$. Then the law $\mathcal{L}(K)$ of the random variable  $ d_K(o,\mathbf{E} \cap \mathbf{L} )$ satisfies $$\mathcal{L}(K)=a_K\delta_{+\infty}+\mathcal{L}_{\rm ac}(K),$$ where $a_K:=\mathbb{P} \left( d_K(o,\mathbf{E} \cap \mathbf{L} \right) = \infty ) = \mathbb{P} ( \mathbf{E} \cap \mathbf{L} = \varnothing ) >0$,  $\delta_{+\infty}$ is the Dirac measure concentrated at $+\infty$ and $\mathcal{L}_{\rm ac}(K)$ is a measure on $(0,\infty)$ which is absolutely continuous with respect to the Lebesgue measure having density
$$f_{d,q,\gamma,K}(\delta) = \frac{D(d,q,\gamma)\omega_{d-\gamma}}{C_{K}(d,d-q+\gamma,u)}  {R(\delta)^{q-\gamma-1}\over\cosh^2(\sqrt{-K}\delta)} \int_0^{1\wedge \frac{ R(u)}{R(\delta)}} z^q (1-z^2)^{\frac{d-q}{2}-1} ( 1+KR(\delta)^2z^2)^{-\frac{d+1}{2}} \,\dd z.$$
\end{theorem}

\begin{figure}[t]
	\centering
	\includegraphics[width=\textwidth]{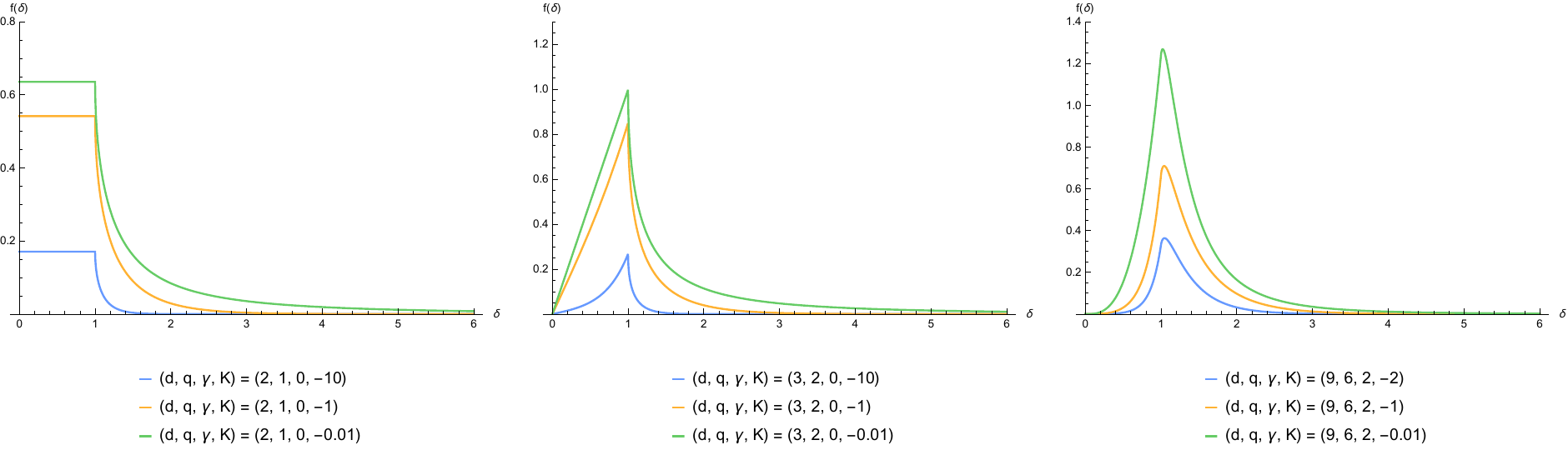}
	\caption{The density $f_{d,q,\gamma,K}$ in Theorem \ref{density} for various choices of the parameters $d$, $q$, $\gamma$ and $K$.}
	\label{fig:densityplot}
\end{figure}

\begin{remark}\label{re:changecurva}
It is possible to transfer the problems treated in Theorems \ref{intersection probability} and \ref{density} from  the space $\mathbb{M}^d_K$ of arbitrary negative curvature $K<0$ to the usual model $\mathbb{H}^d$ of curvature $-1$ by choosing a suitable radius $u$ for $[uB_{-1}]$. More precisely, denote $p_{K,u}$ the probability that $\mathbf{E}$ and $\mathbf{L}$, as above, in $\mathbb{M}^d_{K}$ intersect. Then $p_{K,u}=p_{-1,v}$ with the radius $v$ given by $v=\sqrt{-K}u$.  Additionally, if we let $f^{(u)}_{d,q,\gamma,K}$ be the density of the random variable $d_K(o,\mathbf{E}\cap \mathbf{L})$, then we obtain the  relation:
$
f^{(u)}_{d,q,\gamma,K}(\delta)
=\sqrt{-K}
f^{(v)}_{d,q,\gamma,-1}(\sqrt{-K}\delta),
$
again with $v=\sqrt{-K}u $. Details are provided at the beginning of Section \ref{sec:phases}.
\end{remark}

As in the Euclidean case investigated in \cite{dkt} we can study moment properties of the random variable $d_K(o,\mathbf{E} \cap \mathbf{L} )$. However, in contrast to that case, all positive moments are infinite for $K<0$, because of the fact that the distribution $\mathcal{L}(K)$ has an atom with positive mass at $+\infty$. {On the other hand, conditionally on the event that $d_K(o,\mathbf{E}\cap\mathbf{L})<\infty$, we can show that all positive moments of the random variable $d_K(o,\mathbf{E} \cap \mathbf{L} )$ are finite.}

\newpage

\begin{cor}\label{moments}
	Let $K<0$.
	\begin{itemize}
	 \item[(i)] It holds $\mathbb{E} [d_K(o,\mathbf{E} \cap \mathbf{L} )^\alpha] < \infty$ if and only if $\alpha \in (\gamma-q, 0]$. 
	 \item[(ii)] It holds $\mathbb{E} [d_K(o,\mathbf{E} \cap \mathbf{L} )^\alpha  \mid d_K(o,\mathbf{E}\cap\mathbf{L})<\infty] < \infty$ for every $\alpha>0$. 
	 \end{itemize}
\end{cor}

Observe that the statement of Corollary \ref{moments} (ii) is in sharp contrast to the Euclidean case, since, by \cite[Corollary 5.2]{dkt}, one has that
$$ \mathbb{E} [d_0(o,\mathbf{E} \cap \mathbf{L} )^\alpha  \mid d_0(o,\mathbf{E}\cap\mathbf{L})<\infty] = \mathbb{E} [d_0(o,\mathbf{E} \cap \mathbf{L} )^\alpha  ]= \infty, \quad \text{ if } \alpha \geq \gamma +1. $$

\begin{figure}[t]
	\centering
	\includegraphics[width=0.6\textwidth]{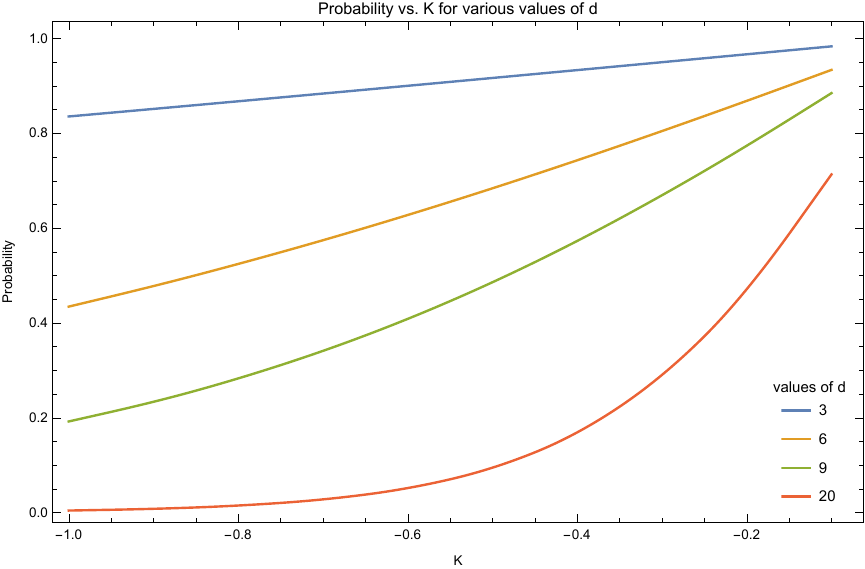}
	\caption{The probability $\mathbb{P} ( \mathbf{E} \cap \mathbf{L} \neq \varnothing )$ plotted against the curvature $K$ for different choices of dimensions $d$.}
	\label{fig:curvatureplot}
\end{figure}

Figure \ref{fig:densityplot} displays the density $f_{d,q,\gamma,K}$ in Theorem \ref{density} for different choices of the dimension parameters $d,q,\gamma$ and curvatures $K$. In particular, it can be observed that for small values of the curvature $K$ the density resembles the one in the Euclidean setting, see \cite[Theorem 5.2 and Figure 5.1]{dkt}. This motivates our next considerations. Namely, we want to demonstrate that the results in the Euclidean case arise from our findings in the limit as the curvature $K$ tends to zero. For the intersection probability $\mathbb{P}(\mathbf{E}\cap \mathbf{L}\neq\varnothing)$ this is illustrated in Figure \ref{fig:curvatureplot}. We recall that $\mathcal{L}(K)$ is the law of the random variable $d_K(o,\mathbf{E}\cap \mathbf{L})$, which for $K=0$ has been investigated in \cite{dkt}.
\begin{theorem}\label{curvaturezero}
It holds that 
$$
\mathcal{L}(K) \longrightarrow \mathcal{L}(0)
$$
weakly, as $K\uparrow 0$. In particular,
$$
\lim_{K\uparrow 0}\mathbb{P} ( \mathbf{E} \cap \mathbf{L} \neq \varnothing ) = 1.
$$
\end{theorem}

Finally, we derive a major consequence of Theorem \ref{intersection probability} by uncovering a phase transition in the study of intersection probabilities in hyperbolic spaces. Thereby we obtain a critical phase which we fully characterize. In view of Corollary \ref{cor:IntersectionProbabdtoInfinity} and Theorem \ref{curvaturezero} one may wonder what is the behavior if simultaneously the curvature tends to zero and the dimension tends to infinity. The next statement addresses this question and illustrates that there are three phases including a critical one.

\begin{theorem}\label{phase transition}
Assume that $K = K(d) < 0$ for every $d \geq 2$. The following holds:
\begin{itemize}
	 \item[(i)]\label{subcritical} If $\lim\limits_{d\to \infty} -K(d) d = 0$ then 
	 $$
\lim_{d\to \infty}\mathbb{P} ( \mathbf{E} \cap \mathbf{L}\neq \varnothing ) = 1.
$$
	 \item[(ii)] \label{supercritical} If $\lim\limits_{d\to \infty}-K(d) d = \infty$ then 
	 $$
\lim_{d\to \infty}\mathbb{P} ( \mathbf{E} \cap \mathbf{L}\neq \varnothing ) = 0.
$$
	 \item[(iii)] \label{critical} If $\lim\limits_{d\to \infty}-K(d) d = \kappa \in (0, \infty)$ then 
	 $$
\lim_{d\to \infty}\mathbb{P} ( \mathbf{E} \cap \mathbf{L}\neq \varnothing ) =  \rho(u,q,\gamma, \kappa) \in (0, 1),
$$
where the constant $\rho(u,q,\gamma, \kappa) $ is given by
\begin{align*}
\rho(u,q,\gamma, \kappa) & = 
\omega_{\gamma+1}   (2\pi )^{-\frac{\gamma+1}{2}}  u^{1+q}{\kappa}^{\gamma+1\over 2}  \Big( \int_0^u e^{\frac{\kappa}{2}s^2} s^{q-\gamma-1} \,\dd s \Big)^{-1} \\
& \quad  \times\int^1_{{0}}   {r^{q-\gamma-1}} 
\int_{0}^{ \frac{1}{r}} 
{v^q}  \exp \Big( - \frac12 u^2 \kappa v^2 (1-r^2) \Big)  \, \dd v
\dd r.\\
\end{align*}
\end{itemize}
\end{theorem}

\begin{figure}[t]
    \centering
    \includegraphics[width=0.8\textwidth]{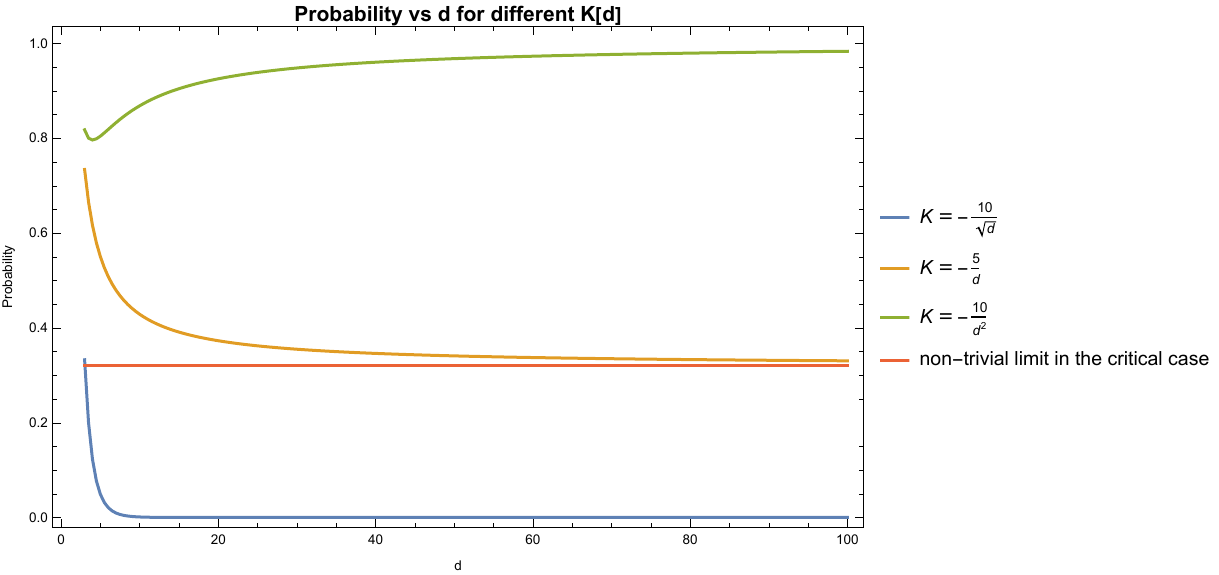}
    \caption{Emergence of three phases for the probability $\mathbb{P} ( \mathbf{E} \cap \mathbf{L} \neq \varnothing )$ if the curvature $K$ tends to zero and the dimension $d$ tends to infinity simultaneously.} 
    \label{fig:phases}
\end{figure}

Figure \ref{fig:phases} visually represents examples of each distinct phase described in Theorem \ref{phase transition}, illustrating how the intersection probability behaves under different conditions on curvature and dimension, alongside the non-trivial constant observed in the critical phase.

\section{The space of constant curvature $K<0$}\label{sec:SpaceK<0}

In this section we provide some preliminary results about the geometry of the space of constant curvature $K<0$. These results are needed for the proofs of our main results and extend those of \cite[Section 5]{bht} from the special case $K=-1$ to general curvature parameters $K$. 

For two points $x=(x_1,\ldots,x_{d+1})$ and $y=(y_1,\ldots,y_{d+1})$ in $\mathbb{R}^{d+1}$ we will use the notation $x\circ y:=x_1y_1+\ldots +x_dy_d-x_{d+1}y_{d+1}$ and $\Vert x \Vert_{L} := x_1^2+\ldots+x_d^2-x_{d+1}^2$ for the Lorentzian inner product and norm, respectively. Moreover, for $x,y\in\RR^d$ with $x=(x_1,\ldots,x_d)$ and $y=(y_1,\ldots,y_d)$ we set $\langle x,y\rangle:= x_1y_1+\ldots+x_dy_d$ and $\| x \| := x_1^2+\ldots +x_d^2$, which are the Euclidean inner product and norm.
A model for the space of constant sectional curvature $K<0$ is provided by one scale of the hyperboloid $\MM_{K}^d=\lbrace x \in \mathbb{R}^{d+1}: \Vert x\Vert_{L} = \frac{1}{K}\rbrace$ endowed with a metric $d_{K}(x,y)$ given by
$$
 x\circ y = K^{-1}\cosh (\sqrt{-K} d_{K}(x,y)),\qquad x,y\in \MM_{K}^d.
 $$
However, it turns out to be more convenient for us to work with the Beltrami-Klein or projective model for $\MM_{K}^d$. For that purpose, let $\mathbb{D}^d_{K}=\lbrace x\in\mathbb{R}^{d+1}: \| x\|^2<-1/K\text{ and } x_{d+1}=0\rbrace$ be the $d$-dimensional open ball of radius $-1/K$, which we consider as being embedded into the hyperplane $\{x=(x_1,\ldots,x_{d+1})\in\RR^{d+1}:x_{d+1}=0\}$ in $\RR^{d+1}$. There is a natural bijection from $\mathbb{D}^d_{K}$ to $\MM_{K}^d$ as defined above. To describe it, let $x=(x_1,\ldots,x_{d+1})\in\mathbb{D}^d_{K}$ and draw a line segment from the origin of $\mathbb{R}^{d+1}$ to $\MM_{-K}^{d}$ that intersects the plane $\Pi=\lbrace (y_1,\ldots,y_d,1/\sqrt{-K}):y_i\in \mathbb{R}\rbrace\subset\RR^{d+1}$ at the point $(x_1,\ldots,x_d,1/\sqrt{-K})$. The line $\epsilon_x=\lbrace (\lambda x_1,\ldots,\lambda x_d,\lambda/\sqrt{-K}):\lambda\in\mathbb{R}\rbrace$ intersects $\MM_{K}^{d}$ at the point that satisfies $\lambda^2(\| x \|^2 + \frac{1}{K})=\frac{1}{K}$, and thus $\lambda=1/(1+K\| x\|^2)^{\frac{1}{2}}$. Therefore, the bijection $p : \mathbb{D}_{K}^{d} \to \MM_{K}^{d}$ is described as follows:
\begin{equation}\label{eq:ProjectionP}
p(x)=\left(\frac{x_1}{(1+K\| x\|^2)^{\frac{1}{2}}},\ldots,\frac{x_d}{(1+K\| x\|^2)^{\frac{1}{2}}}, \frac{1}{\sqrt{-K} (1+K\| x\|^2)^{\frac{1}{2}}}\right)=\frac{x+\frac{e_{d+1}}{\sqrt{-K}}}{(1+K\| x\|^2)^{\frac{1}{2}}}
\end{equation}
in which $e_{d+1}=(0,\ldots,0,1)\in\RR^{d+1}$ is the $(d+1)$-st vector from the standard orthonormal basis in $\RR^{d+1}$.
Note that $(1+K\| x\|^2)^{\frac{1}{2}}=\sqrt{-K}\vert \Vert x+\frac{e_{d+1}}{\sqrt{-K}}\Vert_{L} \vert$ in terms of the Lorentzian norm.
The inverse function of $p$ is
$$
p^{-1}(x)=\left(\frac{x_1}{\sqrt{-K} x_{d+1}},\ldots,\frac{x_d}{\sqrt{-K}x_{d+1}},0\right).
$$

\begin{figure}[t]
  \centering
  \begin{subfigure}[b]{0.35\textwidth}
    \includegraphics[width=\textwidth]{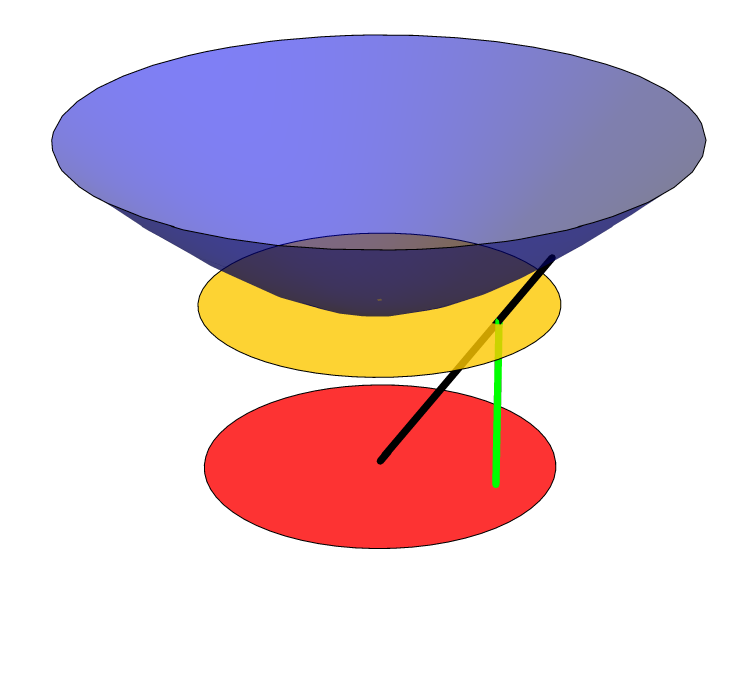}
    \label{fig:BeltramiKlein}
  \end{subfigure}
  \hfill
  \begin{subfigure}[b]{0.4\textwidth}
    \includegraphics[width=\textwidth]{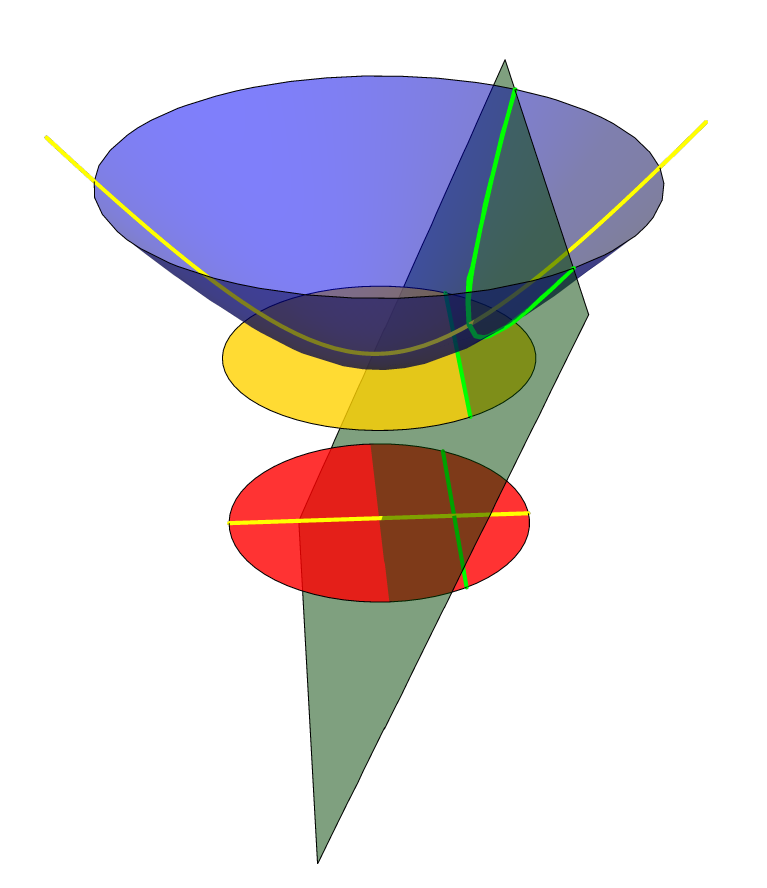}
  \end{subfigure}
  \caption{The blue surface illustrates the hyperboloid $\mathbb{M}^d_K$ in $\mathbb{R}^{d+1}$ and the red disk illustrates the projection onto $\mathbb{D}^d_K$. On the left image there is an example of the projection $p$. On the right image we illustrate how flats appear, the green curves are a flat in generic position $\mathbb{M}^d_K$ and its projection to $\mathbb{D}^d_K$, the yellow curves are a flat and its projection that pass through the origin of each model.}
\end{figure}

For two points $x,y\in\mathbb{D}^d_{K}$ we define their distance as $d_{K}(p(x),p(y))$ and continue to denote this by $d_K(x,y)$, slightly abusing notation. Then, by definition of the metric $d_K$, we obtain 
\begin{align}
\cosh(\sqrt{-K} d_K(x,y))&= \cosh(\sqrt{-K} d_{K}(p(x),p(y))\nonumber \\
&=Kp(x)\circ p(y)\nonumber \\
&=K \frac{x+\frac{e_{d+1}}{\sqrt{-K}}}{(1+K\| x\|^2)^{\frac{1}{2}}}\circ \frac{y+\frac{e_{d+1}}{\sqrt{-K}}}{(1+K\| y\|^2)^{\frac{1}{2}}}\nonumber \\
&=\frac{1+K\langle x,y\rangle }{(1+K\| x\|^2)^{\frac{1}{2}}(1+K\| y\|^2)^{\frac{1}{2}}}. \label{eq:cosh}
\end{align}
Following the structure of Theorem 6.1.5 in \cite{r}, we will determine now the element of arc length for the model $\mathbb{D}^d_{K}$ of $\MM_{K}^d$. Let $y:= p(x)$ and denote $\Vert \mathrm{d}y \Vert_{L}= ( \dd y_1^2 + \ldots +\dd y_d^2 - \dd y_{d+1}^2)^{\frac{1}{2}}$. We calculate the differentials $\dd y_i$ for $i\in\{1,\ldots,d\}$ and separately $\dd y_{d+1}$:
\begin{align*}
\dd y_i&=\sum_{j=1}^d \frac{\partial y_i}{\partial x_j}\dd x_j
\\
&=\frac{1}{(1+K\| x\|^2)^{\frac{1}{2}}}\dd x_i +
\sum_{j=1}^d \frac{x_i(-1/2)2Kx_j \dd x_j}{(1+K\| x\|^2)^{\frac{1}{2}}}\\
&=\frac{\dd x_i}{(1+K\| x\|^2)^{\frac{1}{2}}} -
  \frac{x_i K \langle x,\dd x\rangle }{(1+K\| x\|^2)^{\frac{3}{2}}},\\
(\dd y_i)^2&=
\frac{\dd x_i^2}{1+K\| x\|^2}+\frac{K^2x_i^2\langle x,\dd x\rangle^2}{(1+K\| x \|^2)^3}-\frac{2K x_i\dd x_i \langle x,\dd x\rangle}{(1+K\| x\|^2)^2}\\
\intertext{and for the last coordinate we obtain}
\dd y_{d+1}&=\frac{1}{\sqrt{-K}} \sum_{i=1}^d \frac{\partial}{\partial x_i}
(1+K\| x\|^2)^{-\frac{1}{2}}
\\
&=\frac{1}{\sqrt{-K}}\sum_{i=1}^d\frac{(-1/2)2Kx_i\dd x_i}{(1+K\| x\|^2)^{\frac{3}{2}}}\\
&=\frac{\sqrt{-K}\langle x,\dd x\rangle}{(1+K\| x\|^2)^{\frac{3}{2}}},\\
(\dd y_{d+1})^2&=-\frac{K\langle x, \dd x\rangle^2}{(1+K\| x\|^2)^3}.
\end{align*}
We conclude that the element of arc length is given by
\begin{align*}
&\sum_{i=1}^d \dd y_i^2 - \dd y_{d+1}^2\\
&=\frac{\| \dd x\|^2}{1+K\| x\|^2}+\frac{K^2\| x\|^2\langle x,\dd x\rangle^2}{(1+K\| x \|^2)^3}-\frac{2K  \langle x,\dd x\rangle^2}{(1+K\| x\|^2)^{3}}-\frac{2K^2\| x\|^2\langle x,\dd x\rangle^2- K\langle x, \dd x\rangle^2}{(1+K\| x\|^2)^3}\\
&=\frac{\| \dd x\|^2}{1+K\| x \|^2}-\frac{1}{(1+K\| x\|^2)^3}(K(1+K\| x\|^2)\langle x,\dd x\rangle^2)\\
&=\frac{\| \dd x\|^2}{1+K\| x\|^2}-\frac{K\langle x,\dd x\rangle^2}{(1+K\| x\|^2)^2}.
\end{align*}
From this we see that the Riemannian metric at $x\in\mathbb{D}_K^{d}$ equals 
$$
g_x(u,v)=\frac{1}{(1+K\| x\|^2)^2}\left( (1+K\| x\|^2) \langle u,v\rangle-K\langle x,u\rangle\langle x , v\rangle\right),
$$
where $u$ and $v$ are two vectors in the tangent space of $\mathbb{D}_K^{d}$ at $x$, which in turn can be identified with $\RR^d$.

We will use the element of arc length for the projective model $\mathbb{D}_K^d$ of the hyperbolic space to translate the $q$-dimensional Hausdorff measure restricted to an element $E\in A_K(d,q)$ in the projective model to the intrinsic Euclidean $q$-dimensional Hausdorff measure of $E\cap\mathbb{D}_K^d$ when $E$ is regarded as an element of $A_0(d,q)$. We thereby follow the steps of Lemma 5.1 in \cite{bht}, which deals with the special case $K=-1$. Let $E\in A_0(d,q)$ be a $q$-dimensional affine subspace of $\RR^d$ having non-empty intersection with $\mathbb{D}_{K}^d$ and let $L$ be the unique linear subspace of $\mathbb{R}^{d}$ such that $E=L+\tau(E)$, where $\tau(E)$ denotes the orthogonal projection of the origin onto $E$. Let $u_1,\ldots,u_q$ be an orthonormal basis for $L$.
The Gramm matrix for the Riemannian metric of the basis is calculated in the same manner as in \cite{bht}. Namely, if $I_q$ is the $(q\times q)$-identity matrix, we have
\begin{align}
G(x)&=\frac{1}{(1+K\| x\|^2)^{2q}}
\det \left( (1+K\| x\|^2) I_q -K(\langle x,u_i\rangle\langle x,u_j\rangle)_{i,j=1}^q \right)\nonumber \\
&=\frac{1}{(1+K\| x\|^2)^{2q}}(1+K\| x\|^2)^{q-1}\left(1+K\|x\|^2-\sum_{i=1}^q K\langle x , u_i\rangle^2\right)\nonumber \\
&=\frac{(1+K\| \tau (E)\|^2)}{(1+K\|x\|^2)^{q+1}}.\nonumber
\end{align} 
Now, consider the linear map $\psi :\mathbb{R}^q\to \mathbb{R}^q$ given by $\psi(z)=1+K\|x\|^2+\sum_{i=1}^q\langle a,z\rangle a$, where $a=(\sqrt{-K}\langle x,u_1\rangle,\ldots,\sqrt{-K}\langle x,u_q\rangle)$. It has the eigenvalue $(1+K\|x\|^2)$ with multiplicity $q-1$ and the eigenvalue $1+K\|x\|^2-\sum_{i=1}^q K\langle x , u_i\rangle^2$ with multiplicity one. Applying now Theorem 3.2.3 in \cite{federergmt} we conclude that the restriction to $E$ of the $q$-dimensional Hausdorff measure $\mathcal{H}_K^q$ with respect to the hyperbolic structure with curvature $K$ is given by 
\begin{equation}\label{eq:Haus}
({\mathcal{H}_{K}^q} \raisebox{0.5ex}{$\llcorner$} E)(\,\cdot\,)=\int_{E} \mathbf{1}{\lbrace x\in \,\cdot\, \rbrace}\frac{(1+K\| \tau (E)\|^2)}{(1+K\|x\|^2)^{q+1}}\,\mathcal{H}_0^q(\dd x),
\end{equation}
where $\mathcal{H}_0^q$ stands for the $q$-dimensional Hausdorff measure with respect to the Euclidean structure corresponding to $E$. 

Our next step will be to express the invariant measure $\mu_{q,K}$ on $A_K(d,q)$ in Euclidean terms. It is an extension of Lemma 5.3 in \cite{bht} to general curvatures $K$.
 
\begin{lemma}\label{lem:subst}
Let $q\in \lbrace0,\ldots, d-1\rbrace$. Then 
\begin{align*}
\mu_{q,K}(\,\cdot\,) &=\int_{A_0(d,q)} \mathbf{1}\Big\{ E\cap B^d_0\Big(\frac{1}{\sqrt{-K}}\Big)\in \,\cdot\, \Big\}(1+K\|\tau(E)\|^2)^{-\frac{d+1}{2}}\,\mu_{q,0}(\dd E)\\
&=\int_{G_0(d,d-q)}\int_{L\cap B^d_0\left(\frac{1}{\sqrt{-K}}\right)}\mathbf{1}\Big\{ (L^\perp + x)\cap B^d_0\Big(\frac{1}{\sqrt{-K}}\Big) \in \,\cdot\,\Big\}\\
&\hspace{4cm}\times (1+K\|x\|^2)^{-\frac{d+1}{2}} \,\mathcal{H}_0^{d-q}(\dd x)\nu_{d-q,0}(\dd L).
\end{align*}
\end{lemma}
\begin{proof}
Recall the representation \eqref{eq:MuqK} of the invariant measure $\mu_{q,K}$. We substitute $\cosh(\sqrt{-K}d(x,o))$ according to \eqref{eq:cosh} and the Hausdorff measure according to the relation given by \eqref{eq:Haus}. This yields
\begin{align*}
\mu_{q,K}(\,\cdot\,)&=\int_{G_0(d,d-p)}\int_{L\cap B^d_0\left(\frac{1}{\sqrt{-K}}\right)}\mathbf{1}\Big\{(L^\perp + x)\cap B^d_0\Big(\frac{1}{\sqrt{-K}}\Big) \in \,\cdot\,\Big\} (1+K\|x\|^2)^{-\frac{q}{2}}\\
&\hspace{3cm}\times \frac{1}{(1+K\|x\|^2)^{\frac{d-q+1}{2}}}(1+K\|\tau(L^\perp+x)\|^2)\,\mathcal{H}_0^{d-q}(\dd x)\nu_{d-q,0}(\dd L)\\
&=\int_{G_0(d,d-q)}\int_{L^\perp\cap B^d_0\Big(\frac{1}{\sqrt{-K}}\Big)}\mathbf{1}\Big\{ (L^\perp+ x)\cap B^d_0\Big(\frac{1}{\sqrt{-K}}\Big) \in \,\cdot\,\Big\}\\
&\hspace{6cm}\times (1+K\|x\|^2)^{-\frac{d+1}{2}}\,
\mathcal{H}_0^{d-q}(\dd x)\nu_{q,0}(\dd L)\\
&=\int_{A_0(d,q)} \mathbf{1}\Big\{ E\cap B^d_0\Big(\frac{1}{\sqrt{-K}}\Big)\in \,\cdot\,\Big\}(1+K\|\tau(E)\|^2)^{-\frac{d+1}{2}}\,\mu_{k,0}(\dd E).
\end{align*}
The proof is thus complete.
\end{proof}

\section{Intersection probability}

In this section we determine the intersection probability of a $(d-q+\gamma)$-dimensional totally geodesic submanifold in $\MM^d_{K}$ with a $q$-dimensional totally geodesic submanifold that passes through the origin $o$ in $\MM^d_{K}$ as considered in Theorem \ref{intersection probability}. In view of what has been developed in the previous section, it is useful for us to work in the projective model $\mathbb{D}_K^d$ for $\MM_K^d$. This allows us to treat a ball of hyperbolic radius $u>0$ as a subset of $\mathbb{R}^d$ and we will now express $u$ in terms of the corresponding Euclidean radius which we denote by $R(u)$. As we shall show, this notation is justified by the fact that the Euclidean radius is precisely given by taking $x=u$ in \eqref{eq:DefRu}. To see this, let $y$ be a point with hyperbolic distance $u$ to the origin $o$. Then $u=d_K(o,y)=d_{K}(p(o),p(y))$, where $p:\mathbb{D}_K^d\to\mathbb{M}_K^d$ is the projection map from \eqref{eq:ProjectionP}. We substitute Equation \eqref{eq:cosh} and obtain $d_{K}(p(o),p(y))=\frac{1}{\sqrt{-K}}\mathrm{arcosh}\left(Kp(o)\circ p(y)\right)=\frac{1}{\sqrt{-K}}\mathrm{arcosh}\frac{1}{\sqrt{1+K\| y\|^2}} $. As we mentioned, $\|y\|^2=R^2$ and thus
$$
\sqrt{-K} u=\mathrm{arcosh}\frac{1}{\sqrt{1+KR(u)^2}}.
$$
Using the identity $\sinh( \mathrm{arcosh}(\alpha))=\sqrt{\alpha^2-1}$ for $\alpha>1$, we conclude that in fact
\begin{equation}\label{eq:DefRu}
R(u)=\frac{\sinh (\sqrt{-K}u)}{\sqrt{-K}(1+\sinh^2(\sqrt{-K}u))^{\frac{1}{2}}}.
\end{equation}
After this preparation, we can now use Lemma \ref{lem:subst} to express the hyperbolic intersection probability in purely Euclidean terms:

\begin{proof}[Proof of Theorem \ref{intersection probability}]
It holds
\begin{align*}
\mathbb{P} ( \mathbf{E} \cap \mathbf{L} \neq \varnothing )
& =  C_{K}(d,d-q+\gamma,u)^{-1} \int_{G_{K}(d,q)} \int_{A_{-K}(d,d-q+\gamma)} \mathbf{1}\{E \cap uB^d_{K} \neq \varnothing\}\\
&\qquad\qquad\times \mathbf{1}\{ E \cap L \neq \varnothing\} \mu_{d-q+\gamma,K} (\dd E) \nu_{q,K} (\dd L)\\
&= C_{K}(d,d-q+\gamma,u)^{-1} \int_{G_{0}(d,q)} \int_{A_{0}(d,d-q+\gamma)} \left( 1+K\| \tau(E)\|^2 \right)^{-\frac{d+1}{2}}\\
&\qquad\qquad\times\mathbf{1}\{ E \cap (1/\sqrt{-K}) B^d_{0}\neq \varnothing \}\mathbf{1}\{ E \cap L \cap (1/\sqrt{-K})B^d_{0} \neq \varnothing\}\\
&\qquad\qquad\times \mathbf{1}\{ E \cap R(u)B^d_{0}\neq \varnothing \}\, \mu_{d-q+\gamma,0} (\dd E) \nu_{q,0} (\dd L).
\end{align*}
In order to apply a Blascke-Petkantschin-type formula as stated in Theorem 4.1 of \cite{dkt}, we separate the integrated function into two parts:
\begin{itemize}
\item the nonnegative measurable function $f:A_0(d,\gamma)\to\RR$ given by $f(E\cap L)= \mathbf{1}\{ E \cap L \cap (1/\sqrt{-K})B^d_{0} \neq \varnothing\}$,
\item the rotation-invariant, nonnegative measurable function $H:A_0(d,\gamma)\to\RR$ defined as $H(E)=\left( 1-\| \tau(E)\|^2 \right)^{-\frac{d+1}{2}}\mathbf{1}\{E \cap (1/\sqrt{-K}) B^d_{0}\neq \varnothing\}\mathbf{1}\{ E \cap R(u)B^d_{0}\neq \varnothing \} $.
\end{itemize} 
By rotation invariance, $H(E)$ only depends on $E$ through $d_0(o,E)=\|\tau(E)\|$, that is $H(E)=H_I(d_0(o,E))$ for some function $H_I:[0,\infty)\to\RR$, which in our case is given by
$$
H_I(d_0(o,E))=\left( 1+K d(o,E)^2 \right)^{-\frac{d+1}{2}}
\mathbf{1}\{ d_0(o,E)\leq 1/\sqrt{-K}\}
\mathbf{1}\{ d_0(o,E)\leq R(u)\}.
$$
Applying now \cite[Theorem 3.1]{dkt} we obtain
\begin{equation}\label{eq:sect}
\begin{split}
\mathbb{P} ( \mathbf{E} \cap \mathbf{L} \neq \varnothing ) &=\frac{D(d,d,\gamma)}{C_{K}(d,d-q+\gamma,u)}\int_{A_0(d,\gamma)}\mathbf{1}\{ E\cap (1/\sqrt{-K})B^d\}\\
&\hspace{5cm}\times d_0(o,E)^{-(d-q)}J_H(d_0(o,E))\,\mu_{\gamma,0}(\dd E),
\end{split}
\end{equation}
where $D(d,q,\gamma)$ is the constant from the statement of Theorem \ref{intersection probability} and the function $J_H:[0,\infty)\to\RR$ is given by
\begin{align}
J_H(r)&=\int_{0}^1 H(rz)z^q(1-z^2)^{\frac{d-q}{2}-1}\,\dd z\nonumber \\
&=\int_{0}^1 (1+Kr^2z^2)^{-\frac{d+1}{2}} z^q(1-z^2)^{\frac{d-q}{2}-1}\mathbf{1}\{rz\leq R(u)  \} \,\dd z\nonumber \\
&=\int_{0}^{1\wedge \frac{R(u)}{r}} (1+Kr^2z^2)^{-\frac{d+1}{2}} z^q(1-z^2)^{\frac{d-q}{2}-1}\,\dd z.\label{eq:J_H}
\end{align} 
Plugging this back into the expression \eqref{eq:sect} for $\mathbb{P} ( \mathbf{E} \cap \mathbf{L} \neq \varnothing )$ and then using the decomposition \eqref{eq:MuqEuclidean} of the motion-invariant measure $\mu_{\gamma,0}$, we obtain 
\begin{align*}
\mathbb{P} ( \mathbf{E} \cap \mathbf{L} \neq \varnothing ) &=
\frac{D(d,q,\gamma)}{C_{K}(d,d-q+\gamma,u)}\int_{G_0(d,\gamma)}\int_{L^\perp}\frac{\mathbf{1}\{(L+x)\cap (1/\sqrt{-K})B_0^d\}}{d_0(o,L+x)^{d-q}}\\
&\hspace{6cm}\times J_H(d_0(o,L+x))\,\lambda_{L^{\perp}}(\dd x)\nu_{\gamma,0}(\dd L).
\end{align*}
Since the integrated function is invariant under rotations, the inner integral does not depend on the choice of $L$. Thus, choosing for $L$ the canonical subspace $\RR^{d-\gamma}\subset\RR^d$, the probability can be written as
\begin{align*}
\mathbb{P} ( \mathbf{E} \cap \mathbf{L} \neq \varnothing ) &=
\frac{D(d,q,\gamma)}{C_{K}(d,d-q+\gamma,u)}\int_{\mathbb{R}^{d-\gamma}}\frac{\mathbf{1}\{ \|x\|< (1/\sqrt{-K}) \}}{\|x\|^{d-q}}J_H(\|x\|)\,\dd x.
\end{align*}
Next, we substitute $J_H$ according to \eqref{eq:J_H}:
\begin{align*}
\mathbb{P} ( \mathbf{E} \cap \mathbf{L} \neq \varnothing ) &=
\frac{D(d,q,\gamma)}{C_{K}(d,d-q+\gamma,u)}\int_{\mathbb{R}^{d-\gamma}}\frac{\mathbf{1}\{\|x\|<(1/\sqrt{-K}) \}}{\|x\|^{d-q}}\\
&\hspace{6cm}\times\int_{0}^{1\wedge \frac{R(u)}{\|x\|}} 
\frac{z^q(1-z^2)^{\frac{d-q}{2}-1}}{(1+Kr^2z^2)^{\frac{d+1}{2}} }\,\dd z
\dd x.
\end{align*}
Finally, we introduce spherical coordinates in $\RR^{d-\gamma}$ and obtain
\begin{align*}
\mathbb{P} ( \mathbf{E} \cap \mathbf{L} \neq \varnothing ) &=
\frac{D(d,q,\gamma)\omega_{d-\gamma}}{C_{K}(d,d-q+\gamma,u)}
\int_{0}^\infty \frac{\mathbf{1}\{r < (1/\sqrt{-K}) \}}{r^{d-q}}\\
&\hspace{6cm}\times\int_{0}^{1\wedge \frac{R(u)}{r}} 
\frac{z^q(1-z^2)^{\frac{d-q}{2}-1}}{(1+Kr^2z^2)^{\frac{d+1}{2}} }r^{d-\gamma-1}\,\dd z\dd r.
\end{align*}
This completes the proof of Theorem \ref{intersection probability}.
\end{proof}
Now we can prove Corollary \ref{cor:IntersectionProbabdtoInfinity}.

\begin{proof}[Proof of Corollary \ref{cor:IntersectionProbabdtoInfinity}]
First note that the constant $C_{K}(d,d-q+\gamma,u)$, by definition and the fact that $\cosh (x) \geq 1$, can be bounded from below by a constant independent of $d$. Moreover, also by definition we have
$$ D(d,q,\gamma)\omega_{d-\gamma} = \frac{\omega_{\gamma +1} \omega_{q-\gamma} \omega_{d-q}}{\omega_{d-q+\gamma+1} }.$$
Using the definition of $\omega_n$ and the Stirling formula, we see that the latter expression is, up to a constant, asymptotically equivalent to $(\frac{d-q+\gamma+1}{2 e})^{\frac{\gamma+1}{2}}$. Thus, by Theorem \ref{intersection probability} it suffices to show that
	$$\lim_{d\to\infty}  \Big( \frac{d-q+\gamma+1}{2 e}\Big)^{\frac{\gamma+1}{2}} \int_{0}^{1/\sqrt{-K}}   {r^{q-\gamma-1}}
\int_{0}^{1\wedge \frac{R(u)}{r}} 
\frac{z^q(1-z^2)^{\frac{d-q}{2}-1}}{(1+Kr^2z^2)^{\frac{d+1}{2}} }\, \dd z
\dd r= 0 .$$
Indeed, in the proof of Theorem \ref{phase transition}, see Equation \eqref{asympsup2}, it is more generally shown that we have
$$\int_{0}^{1/\sqrt{-K}} \int_{0}^{1\wedge \frac{R(u)}{r}} 
\frac{z^q(1-z^2)^{\frac{d-q}{2}-1}}{(1+Kr^2z^2)^{\frac{d+1}{2}} }\, \dd z
\dd r= O(d^{-\frac{q+1}{2}}), \quad d \to \infty.$$
Since $q >\gamma$, this shows that 
$$\lim_{d\to\infty}  \Big( \frac{d-q+\gamma+1}{2 e}\Big)^{\frac{\gamma+1}{2}} \int_{0}^{1/\sqrt{-K}}   {r^{q-\gamma-1}}
\int_{0}^{1\wedge \frac{R(u)}{r}} 
\frac{z^q(1-z^2)^{\frac{d-q}{2}-1}}{(1+Kr^2z^2)^{\frac{d+1}{2}} }\, \dd z
\dd r= 0,$$
as desired.
\end{proof}

\section{Density of the distance}

Let $d_{K}(o,\mathbf{E}\cap \mathbf{L})$ be the distance from $o$ of the intersection of $\mathbf{E}$ with $\mathbf{L}$. We know from Theorem \ref{intersection probability} that the distribution of this random variable has an atom at $+\infty$. The purpose is now to prove that on the interval $(0,\infty)$ the law of $d_{K}(o,\mathbf{E}\cap \mathbf{L})$ is absolutely continuous with respect to the Lebesgue measure and to determine its density. For that purpose, we will calculate the probability that this distance is smaller than $\delta \in (0, \infty)$. We follow essentially the same procedure as before, express the invariant measure $\mu_{q,K}$ according to Lemma \ref{lem:subst} in Euclidean terms and then use the Blaschke-Petkantschin-type formula from \cite{dkt}. 

\begin{proof}[Proof of Theorem \ref{density}]
Recalling the definition of $R(\,\cdot\,)$ from \eqref{eq:DefRu}, we can write
\begin{align}
&\mathbb{P} ( d_{K}(o, \mathbf{E} \cap \mathbf{L}) \leq \delta )\nonumber\\
& =  C_{K}(d,d-q+\gamma,u)^{-1} \int_{G_{K}(d,q)} \int_{A_{K}(d,d-q+\gamma)} 
\mathbf{1}\{E \cap uB^d_{K} \neq \varnothing \}\nonumber\\
&\hspace{6cm}
\times\mathbf{1}\{d_{K}(o,E\cap L)\leq \delta \}\, \mu_{d-q+\gamma,K} (\dd E) \nu_{q,K} (\dd L)\nonumber \\
&= C_{K}(d,d-q+\gamma,u)^{-1} \int_{G_{0}(d,q)} \int_{A_{0}(d,d-q+\gamma)} \left( 1+K\| \tau(E)\|^2 \right)^{-\frac{d+1}{2}}
\mathbf{1}\{ E \cap (1/\sqrt{-K}) B^d_{0}\neq \varnothing \}
\nonumber \\
&\hspace{2cm}\times \mathbf{1}\{d_0(o,E\cap L)\leq R(\delta)\}\mathbf{1}{\lbrace E \cap R(u)B^d_{0}\neq \varnothing \rbrace}\, \mu_{d-q+\gamma,0} (\dd E) \nu_{q,0} (\dd L)\nonumber \\
&=\frac{D(d,q,\gamma)}{C_{K}(d,d-q+\gamma,u)}\int_{A_0(d,\gamma)}
\mathbf{1}{\lbrace d_{0}(o,E)\leq \delta \rbrace}
d_0(o,E)^{-(d-q)}J_H(d_0(o,E))\,\mu_{\gamma,0}(\dd E).\label{eq:distrib}
\end{align}
In the last expression the function $H$ is the same as in the prevision section and, consequently, the function $J_H$ is also the same as before, recall \eqref{eq:J_H}. What differs is the function $f$, which this time is given by $f(E\cap F)=\mathbf{1}{\lbrace d_{0}(o,E\cap L)\leq R(\delta) \rbrace}$. Therefore we can substitute $J_H$ from \eqref{eq:J_H} to \eqref{eq:distrib} and obtain
\begin{align}
\mathbb{P} ( d_{K}(o, \mathbf{E} \cap \mathbf{L}) \leq \delta )
& =  
\frac{D(d,q,\gamma)}{C_{K}(d,d-q+\gamma,u)}
\int_{A_0(d,\gamma)}\frac{\mathbf{1}{\lbrace d_0(o,E)\leq R(\delta) \rbrace}}{d_0(o,E)^{d-q}}
\nonumber \\
&\hspace{3cm}\times \int_{0}^{1} 
\frac{z^q(1-z^2)^{\frac{d-q}{2}-1}}{(1+Kr^2z^2)^{\frac{d+1}{2}} } \mathbf{1}\{ d_0(o,E)z\leq R(u) \} \,\dd z\mu_{\gamma,0}(\dd E).\label{eq:distrib2}
\end{align}
Next, we substitute \eqref{eq:MuqEuclidean} into \eqref{eq:distrib2} and, as in the previous section, use the fact that the resulting expression is independent of the choice of the linear subspace $L$, which we then choose as $\RR^{d-\gamma}$. We obtain
\begin{align}
\mathbb{P} ( d_{K}(o, \mathbf{E} \cap \mathbf{L}) \leq \delta )
& =  
\frac{D(d,q,\gamma)}{C_{K}(d,d-q+\gamma,u)}\int_{\mathbb{R}^{d-\gamma}}\frac{\mathbf{1}{\lbrace \|x\|\leq R(\delta) \rbrace}}{\|x\|^{d-q}}\\
&\hspace{4cm}\times \int_{0}^{1\wedge \frac{R(u)}{\|x\|}} 
\frac{z^q(1-z^2)^{\frac{d-q}{2}-1}}{(1+Kr^2z^2)^{\frac{d+1}{2}} }\,\dd z
\dd x.\nonumber
\end{align}
Transforming the last expression to spherical coordinates yields
\begin{align}
\mathbb{P} ( d_{K}(o, \mathbf{E} \cap \mathbf{L}) \leq \delta )
& =  
\frac{D(d,q,\gamma)\omega_{d-\gamma}}{C_{K}(d,d-q+\gamma,u)}\int_{0}^{R(\delta)} r^{-(d-q)}
\int_{0}^{1\wedge \frac{R(u)}{r}} 
\frac{z^q(1-z^2)^{\frac{d-q}{2}-1}}{(1+Kr^2z^2)^{\frac{d+1}{2}} }\, \dd z r^{d-\gamma-1}
\,\dd r\nonumber\\
&=  
\frac{D(d,q,\gamma)\omega_{d-\gamma}}{C_{K}(d,d-q+\gamma,u)}\int_{0}^{R(\delta)} r^{(q-\gamma-1)}
\int_{0}^{1\wedge \frac{R(u)}{r}} 
\frac{z^q(1-z^2)^{\frac{d-q}{2}-1}}{(1+Kr^2z^2)^{\frac{d+1}{2}} }\,\dd z\,\dd r.\label{re:density}
\end{align}
Differentiation of \eqref{re:density} with respect to $\delta$ together with the fact that ${\dd R(\delta)\over\dd\delta}={\cosh^{-2}(\sqrt{-K}\delta)}$ for $\delta>0$ complete the proof of Theorem \ref{density}.\end{proof}

Now we will prove Corollary \ref{moments}.

\begin{proof}[Proof of Corollary \ref{moments}]
Note that for $\alpha >0$ it holds $d(p,\mathbf{E} \cap \mathbf{L} )^\alpha = \infty$ with positive probability. Now let $\alpha <0$. Then the claim follows from Theorem \ref{density}, by checking the integrability of the function $\delta^\alpha f_{d,q,\gamma,K}(\delta)$ at 0. Indeed, near 0 the function $f_{d,q,\gamma,K}(\delta)$ is, up to a constant, asymptotically equivalent to $\delta^{q-\gamma-1}$, which follows from the definition of $R(\delta)$ and by using the Taylor expansion of the functions $\cosh(x)$ and $\sinh(x)$ at $x=0$. This proves the first statement of the Corollary. For the second statement it suffices to note that $f_{d,q,\gamma,K}(\delta)$ decays exponentially as $\delta\to\infty$.
\end{proof}

\section{Behaviour as the curvature tends to $0$}\label{sec:CurvatureToZero}

In this section we study the behaviour of the intersection probabilities as the curvature $K<0$ of the space $\MM_K^{d}$ tends to $0$. 
For every $K<0$ the space of constant sectional curvature $\MM_{K}^d$ is defined in $\mathbb{R}^{d+1}$ as a hyperboloid given by the elements $x=(x_1,\ldots,x_{d+1})$ satisfying $x_{d+1}>0$ and the equation $\Vert x\Vert_L^2=\frac{1}{K}$. As $K$ increases to $0$ the hyperboloid flattens and moves away from $0$ in $\mathbb{R}^{d+1}$. The corresponding projective model $\mathbb{D}_{K}^d$ is defined inside the Euclidean ball of radius $1/\sqrt{-K}$, which, as $K$ tends to $0$, grows to infinity. Intuitively one would expect that the distances in the projective model will converge to distances in the Euclidean space. Moreover, the expression in Euclidean distance of the radius of a ball in the projective model should converge pointwise to the identity map as the curvature $K$ tends to $0$. This is exactly what happens, precisely, $R_{K}(u)$ tends to $u$ as $K$ approaches $0$. Recall that
$$
R_{K}(u)  := R(u)  =\frac{\sinh (\sqrt{-K}u)}{\sqrt{-K}(1+\sinh^2(\sqrt{-K}u))^{\frac{1}{2}}}.
$$
Both $\sinh(\sqrt{-K}u)$ and $\sqrt{-K}(1+\sinh^2(\sqrt{-K}u))^{\frac{1}{2}}$ tend to $0$ as $K$ tends to $0$. By L'H\^{o}pital's rule 
and a substitution of $t=\sqrt{-K}$, we get
\begin{align*}
\lim_{K\uparrow 0} R_{K}(u)&=\lim_{K\uparrow 0}\frac{\sinh (\sqrt{-K}u)}{\sqrt{-K}(1+\sinh^2(\sqrt{-K}u))^{\frac{1}{2}}}\\
&=\lim_{t\downarrow 0}\frac{u\cosh (tu)}{(1+\sinh^2(tu))^\frac{1}{2}(1+tu\sinh (tu)\cosh (tu))}=u,\qquad u>0.
\end{align*}

We turn now to the intersection probability of the flats $\mathbf{E}$ and $\mathbf{L}$. As we calculated in the previous section, the probability that a $k$-flat intersects a $q$-flat that passes through the origin is not trivial in the case of constant negative curvature $K$, but it is trivial in the case of the Euclidean space, where the curvature is $0$. It is natural to ask whether the probability tends to $1$ as the curvature is flattened. This turns out to be true and is an immediate consequence of the first statement in Theorem \ref{curvaturezero}.

\begin{proof}[Proof of Theorem \ref{curvaturezero}]
Recall that $\mathcal{L}(K)$ is the law of the random variable $d_K(o,\mathbf{E}\cap \mathbf{L})$. In order to prove the weak convergence of $\mathcal{L}(K)$ to $\mathcal{L}(0)$, as $K\uparrow 0$, we will prove the following limit result: $$
\lim_{K\uparrow  0}
\mathbb{P}_{K}(d_{K}(o, \mathbf{E}\cap \mathbf{L})\leq \delta ) =
\mathbb{P}_{0}(d_0(o,\mathbf{E}\cap \mathbf{L})\leq \delta),\qquad\delta\in(0,\infty).
$$
Here, the notation $\mathbb{P}_K$ indicates that the random flats $\mathbf{E}$ and $\mathbf{L}$ are chosen in $\mathbb{M}_K^d$.
In the previous sections we expressed the cumulative density function $  
\mathbb{P}_K ( d_{K}(o, \mathbf{E} \cap \mathbf{L}) \leq \delta )$ as
$$  
\mathbb{P}_K ( d_{K}(o, \mathbf{E} \cap \mathbf{L}) \leq \delta )
=  
\frac{D(d,q,\gamma)\omega_{d-\gamma}}{C_{K}(d,d-q+\gamma)}\int_{0}^{R_K(\delta)} r^{(q-\gamma-1)}
\int_{0}^{1\wedge \frac{R_{K}(u)}{r}} 
\frac{z^q(1-z^2)^{\frac{d-q}{2}-1}}{(1+Kr^2z^2)^{\frac{d+1}{2}} }\,\dd z\dd r .
$$

We will separate the first integral in two parts and apply the theorem of dominated convergence in each part. For simplicity in the calculations we define:
\begin{align*}
F_K(r,z) &:=\mathbf{1}\lbrace  z\leq   1\wedge R_{K}(u)/r \rbrace   (1+Kr^2z^2)^{-\frac{d+1}{2}} z^q(1-z^2)^{\frac{d-q}{2}-1} 
\intertext{and}
G(r,z) &:=\mathbf{1}\lbrace z\leq 1\wedge u/ r \rbrace  z^q(1-z^2)^{\frac{d-q}{2}-1}.
\end{align*}

Let $r\geq u$, we remark that $u \geq R_K(u) $ for any choice of $K$, therefore $r\geq u \geq R_K(u)$.
The numbers $r,z$, that satisfy the restrictions of the inner integral, must satisfy $rz\leq r\frac{R_K(u)}{r}=R_K(u)$ and consequently $-Kr^2z^2\leq -KR^2_K(u)$. Since
$$
\lim_{K\uparrow 0}KR_{K}^2(u) = \lim_{K\uparrow 0}\frac{\sinh^2(u\sqrt{-K})}{1+\sinh^2(u\sqrt{-K})}=0,
$$
we can assume that $-Kr^2z^2\leq 1/2$ for any such a choice of $r,z$, if $K$ is sufficiently close to $0$. Consequently $(1+Kr^2z^2)^{-\frac{d+1}{2}}\leq 2^{\frac{d+1}{2}}$, which is a constant bound independent of $r$ and $z$. We combine these inequalities to obtain
\begin{align*}
\mathbf{1}\lbrace  z\leq   R_{K}(u)/r \rbrace  (1+Kr^2z^2)^{-\frac{d+1}{2}} z^q(1-z^2)^{\frac{d-q}{2}-1} \leq 
2^{\frac{d+1}{2}} \mathbf{1}\lbrace z\leq 1\wedge u/ r \rbrace   z^q(1-z^2)^{\frac{d-q}{2}-1} .
\end{align*}
We achieved $
F_K(r,z)\leq 2^{\frac{d+1}{2}}G(r,z).
$
It is obvious that
$\lim_{K\uparrow 0} F_K(r,z)=G(r,z)$ and that
$
\int_0^1 2^{\frac{d+1}{2}} G(r,z)\,\dd z <\infty
$. By the dominated convergence theorem, for every $r>u$ we obtain
$$
\lim_{K\uparrow 0} \int_0^1 F_K(r,z) \,\dd z= \lim_{K\uparrow 0}\int_{0}^{1\wedge \frac{R_{K}(u)}{r}} 
\frac{z^q(1-z^2)^{\frac{d-q}{2}-1}}{(1+Kr^2z^2)^{\frac{d+1}{2}} }\,\dd z= \int_0^{1\wedge \frac{u}{r}}z^q (1-z^2)^{\frac{d-q}{2}-1}\,\dd z.
$$

For $r<u$, there is a $K'\in (-1,0)$ such that $R_K(u)>r$ for every $K>K'$, since $R_K(u)\to u$. Moreover  $KR_K^2$ tends to $0$, therefore $K'R_{K'}^2< K R_K^2<Kr^2$. It follows that
\begin{align}\label{eq:int}
\int_{0}^{1\wedge \frac{R_{K}(u)}{r}} 
\frac{z^q(1-z^2)^{\frac{d-q}{2}-1}}{(1+Kr^2z^2)^{\frac{d+1}{2}} }\,\dd z=\int_{0}^{1 } 
\frac{z^q(1-z^2)^{\frac{d-q}{2}-1}}{(1+Kr^2z^2)^{\frac{d+1}{2}} }\,\dd z<\int_{0}^{1 } 
\frac{z^q(1-z^2)^{\frac{d-q}{2}-1}}{(1+K'R_{K'}^2z^2)^{\frac{d+1}{2}} }\,\dd z<\infty.
\end{align}
Applying one more time the dominated convergence theorem, we obtain for every $r>0$,
 $$\lim_{K\uparrow 0} \int_0^1 F_K(r,z)\,\dd z=\int_0^1 G(r,z)\,\dd z.$$

Next step is to repeat similar arguments for the second integration. We define $M_K(r):=\mathbf{1}\lbrace r \leq R_K(\delta) \rbrace r^{q-\gamma-1} \int_0^1 F_K(r,z)\,\dd z$ and $N(r):=\mathbf{1}\lbrace r\leq \delta\rbrace r^{q-\gamma-1} \int_0^1 G(r,z)\,\dd z $. Recall that $R_K(\delta)\leq \delta$ for every $\delta>0$. The previous calculations showed that $M_K$ converges to $N$ pointwise as $K$ tends to $0$. The only difficulty in this case is that we split the integral in different regions depending on $K$. 
 
For $r\geq R_K(u)$ we have already shown that $M_K(r)\leq 2^{\frac{d+1}{2}} N(r)$ for $K$ sufficiently close to $0$ uniformly for all $r$. Next, we integrate $2^{\frac{d+1}{2}}N(r)$ and obtain
$$
\int_{R_K(u)}^{\infty}2^{\frac{d+1}{2}}N(r)\,\dd r \leq 2^{\frac{d+1}{2}}\int_0^{\infty}r^{q-\gamma-1}\int_{0}^{1\wedge \frac{u}{r}}  z^q(1-z^2)^{\frac{d-q}{2}-1}\,\dd z \dd r= 2^{\frac{d+1}{2}}\frac{C_0(d,d-q+\gamma)}{D(d,q,\gamma)\omega_{d-\gamma}}<\infty.
$$
For $r<R_K(u)$, as before $KR^2_K$ tends to $0$, thus $K'R^2_{K'}<KR^2_K<Kr^2$ for every large $K$ and 
\begin{align*}
\int_{0}^{R_K(u)} 
M_K(r) \, \dd r  &\leq \int_{0}^{R_K(u)} 
r^{q-\gamma-1}
\int_{0}^{1} 
\frac{z^q(1-z^2)^{\frac{d-q}{2}-1}}{(1+Kr^2z^2)^{\frac{d+1}{2}} }\,\dd z\dd r \\
&\leq 
\int_{0}^{R_K(u)} 
r^{q-\gamma-1}
\int_{0}^{1} 
\frac{z^q(1-z^2)^{\frac{d-q}{2}-1}}{(1+K'R_{K'}^2(u)z^2)^{\frac{d+1}{2}} }\,\dd z\dd r\\
&\leq 
\int_{0}^u
r^{q-\gamma-1}
\, \dd r
\int_{0}^{1} 
\frac{z^q(1-z^2)^{\frac{d-q}{2}-1}}{(1+K'R_{K'}^2(u)z^2)^{\frac{d+1}{2}} }\,\dd z<\infty.
\end{align*}
 Combining these bounds, we can apply the dominated convergence theorem once again to conclude that
\begin{align*}
&\lim_{K\uparrow 0} \int_{0}^{R_K(\delta)} r^{q-\gamma-1} \int_{0}^{1\wedge \frac{R_{K}(u)}{r}} 
 \frac{z^q(1-z^2)^{\frac{d-q}{2}-1}}{(1+Kr^2z^2)^{\frac{d+1}{2}} }\, \dd z\dd r
  =\lim_{K\uparrow 0} \int_{0}^{\infty} M_K(r)\,\dd r\\
 &= \int_{0}^{\infty} N(r)\,\dd r = \int_{0}^{\delta} r^{q-\gamma-1}\int_0^{1\wedge \frac{u}{r}}z^q (1-z^2)^{\frac{d-q}{2}-1}\,\dd z \dd r.
\end{align*}
To complete the proof we eventually need to prove that $\lim_{K\uparrow 0 } C_{-K}(d,d-q+\gamma,u) = C_{0}(d,d-q+\gamma,u)$ for $u>0$. However, this easily follows from \eqref{normconstant}, by using a Taylor expansion of the functions $\cosh(x)$ and $\sinh(x)$ at $x=0$.
\end{proof}

\section{Phase transition in terms of the curvature and dimension} \label{sec:phases}

First, we provide the calculations that are necessary to justify the claims we made in Remark \ref{re:changecurva}. 

\begin{proof}[Proof of Remark \ref{re:changecurva}]
First we observe the following relation about the function $R$ and different curvatures: $\sqrt{-K} R_K(u)= R_{-1}(\sqrt{-K} u) $ for $u>0$ and $K<0$. Similarly, we have the following relation about the normalization constant: $  C_K(d,d-q+\gamma,u )=(\sqrt{-K})^{-q+\gamma}C_{-1}(d,d-q+\gamma,\sqrt{-K} u)$. In fact, applying the substitution $r=s\sqrt{-K} $ in \eqref{normconstant}, we see that
\begin{align}\label{eq:constantC}
C_K\left(d,q,u \right)&= \omega_{d-q}(\sqrt{-K})^{-d+q+1} \int_0^{u} \cosh^{q}( \sqrt{-K}s)\sinh^{d-q-1}(\sqrt{-K} s) \, \dd s\\ \nonumber
&=(\sqrt{-K})^{-d+q} \omega_{d-q} \int_0^{u} \cosh^{q}( \sqrt{-K}s)\sinh^{d-q-1}(\sqrt{-K} s) (\sqrt{-K})\, \dd s\\ \nonumber
&=(\sqrt{-K})^{-d+q} \omega_{d-q} \int_0^{\sqrt{-K}u} \cosh^{q}( r)\sinh^{d-q-1}( r) \,\dd r \\ \nonumber
&=(\sqrt{-K})^{-d+q}C_{-1}(d,q,u).
\end{align}
We can now express $p_{K,u}$ as integral according to Theorem \ref{intersection probability}, we use Equation \eqref{eq:constantC} and once more we apply the substitution $r=s\sqrt{-K}$. Using the notation $\mathbb{P}_K^{(u)}$ to indicate that the random flats $\mathbf{E}$ and $\mathbf{L}$ are chosen in $\mathbb{M}_K^d$ and that $\mathbf{E}\cap u{B}_K^d\neq\varnothing$, this yields
\begin{align*}
p_{K,u}&=\mathbb{P}^{(u)}_{K} (  \mathbf{E} \cap \mathbf{L}\neq \varnothing )\\
&=  \frac{D(d,q,\gamma)\omega_{d-\gamma}}{  C_K(d,d-q+\gamma,u)}
\int_{0}^{1/\sqrt{-K}}  s^{(q-\gamma-1)}
\int_{0}^{1\wedge \frac{ R_{K}(u)}{ s}} 
\frac{z^q(1-z^2)^{\frac{d-q}{2}-1}}{(1-(\sqrt{-K} s)^2z^2)^{\frac{d+1}{2}} }\, \dd z \dd s\\
&=  
\frac{D(d,q,\gamma)\omega_{d-\gamma}}{(\sqrt{-K})^{q-\gamma}  C_K(d,d-q+\gamma,u)}\int_{0}^{1/\sqrt{-K}} (  \sqrt{-K}s)^{(q-\gamma-1)}\\
&\hspace{4cm}\times
\int_{0}^{1\wedge \frac{\sqrt{-K} R_{K}(u)}{\sqrt{-K} s}} 
\frac{z^q(1-z^2)^{\frac{d-q}{2}-1}}{(1-(\sqrt{-K} s)^2z^2)^{\frac{d+1}{2}} }\sqrt{-K} \, \dd z \dd s\\
&=  
\frac{D(d,q,\gamma)\omega_{d-\gamma}}{C_{-1}(d,d-q+\gamma,u)}\int_{0}^{1} r^{(q-\gamma-1)}
\int_{0}^{1\wedge \frac{ R_{-1}(\sqrt{-K}u)}{ r}} 
\frac{z^q(1-z^2)^{\frac{d-q}{2}-1}}{(1-r^2z^2)^{\frac{d+1}{2}} } \, \dd z \dd r\\
&=\mathbb{P}^{(\sqrt{-K}u)}_{-1} (  \mathbf{E} \cap \mathbf{L}\neq \varnothing )=p_{-1,v}.
\end{align*}
We conclude the calculations by repeating the same arguments for the density function of the distance of $\mathbf{E}\cap\mathbf{L}$ to $o$:
\begin{align*}
&\mathbb{P}_{K}^{(u)} ( d_{K}(o, \mathbf{E} \cap \mathbf{L}) \leq \delta )\\
&=  
\frac{D(d,q,\gamma)\omega_{d-\gamma}}{C_K(d,d-q+\gamma,u)}\int_{0}^{R_{K}(\delta)} s^{(q-\gamma-1)}
\int_{0}^{1\wedge \frac{R_{K}(u)}{s}} \frac{z^q(1-z^2)^{\frac{d-q}{2}-1}}{(1-(\sqrt{-K} s)^2z^2)^{\frac{d+1}{2}} }\, \dd z \dd s\\
&=  
\frac{D(d,q,\gamma)\omega_{d-\gamma}}{(\sqrt{-K})^{q-\gamma}  C_K(d,d-q+\gamma,u)}
\int_{0}^{ \sqrt{-K} R_{K}(\delta) } r^{(q-\gamma-1)}
\int_{0}^{1\wedge \frac{\sqrt{-K}R_{K}(u)}{r}} 
\frac{z^q(1-z^2)^{\frac{d-q}{2}-1}}{(1-r^2z^2)^{\frac{d+1}{2}} }\, \dd z  \dd r\\
&=\frac{D(d,q,\gamma)\omega_{d-\gamma}}{C_{-1}(d,d-q+\gamma,u)}
\int_{0}^{ R_{-1}( \sqrt{-K} \delta )} r^{(q-\gamma-1)}
\int_{0}^{1\wedge \frac{R_{-1}(\sqrt{-K}u )}{r}} 
\frac{z^q(1-z^2)^{\frac{d-q}{2}-1}}{(1-r^2z^2)^{\frac{d+1}{2}} }\, \dd z  \dd r\\
&=\mathbb{P}_{-1}^{(v)} ( d_{-1}(o, \mathbf{E} \cap \mathbf{L}) \leq \sqrt{-K}\delta ).
\end{align*}
This completes the proof.
\end{proof}

Finally, we give a proof of Theorem \ref{phase transition}. In particular, we use the notation $a(d)\sim b(d)$ for two real-valued sequences $a(d)$ and $b(d)$ depending on a parameter $d$ to indicate that $a(d)/b(d)\to 1$ as $d\to\infty$. Moreover, we use the usual Landau notation $a(d)=O(b(d))$ if $\limsup_{d \to \infty}|a(d)/b(d)|<\infty$ and $a(d)=o(b(d))$ provided that $\lim_{d\to\infty}|a(d)/b(d)|=0$.

\begin{proof}[Proof of Theorem \ref{phase transition}]

We start with the subcritical phase. That it, we assume that
$$ - K (d) d \to 0, \quad d \to \infty.$$
Recall that
\begin{align*}
& \mathbb{P}^{(u)}_{K} (  \mathbf{E} \cap \mathbf{L}\neq \varnothing )=  \frac{D(d,q,\gamma)\omega_{d-\gamma}}{  C_K(d,d-q+\gamma,u)}
\int_{0}^{1/\sqrt{-K}}  r^{q-\gamma-1}
\int_{0}^{1\wedge \frac{ R_{K}(u)}{ r}} 
\frac{z^q(1-z^2)^{\frac{d-q}{2}-1}}{(1-(\sqrt{-K} r)^2z^2)^{\frac{d+1}{2}} }\, \dd z \dd r.
\end{align*}
Now we calculate asymptotic expressions, for $d \to \infty$, step by step.

\paragraph{\textbf{Step 1:}}
Similarly as argued in the proof of Corollary \ref{cor:IntersectionProbabdtoInfinity} by an application of Stirling's formula we have
\begin{equation}\label{asympsub1}
D(d,q,\gamma)\omega_{d-\gamma} \sim  \omega_{\gamma+1} \omega_{q-\gamma}  \Big( \frac{1}{2\pi } \Big)^{\frac{\gamma+1}{2}} d^{\frac{\gamma+1}{2}}, \quad d \to \infty .
\end{equation}
\paragraph{\textbf{Step 2:}} By observations already made and the fact that  $- K (d) d \to 0$, $d \to \infty$, we have
 \begin{equation}\label{asympsub2}
C_{K}(d,d-q+\gamma, u) \sim  \frac{\omega_{q-\gamma}}{q-\gamma}u^{q-\gamma} , \quad d \to \infty .
\end{equation}
\paragraph{\textbf{Step 3:}} It remains to consider the asymptotics of the expression
$$\int_{0}^{1/\sqrt{-K}}    {r^{q-\gamma-1}}
\int_{0}^{1\wedge \frac{R_K(u)}{r}} 
\frac{z^q(1-z^2)^{\frac{d-q}{2}-1}}{(1+Kr^2z^2)^{\frac{d+1}{2}} }\, \dd z
\dd r.$$
We use the estimate
$$\int_{0}^{1/\sqrt{-K}}    {r^{q-\gamma-1}}
\int_{0}^{1\wedge \frac{R_K(u)}{r}} 
\frac{z^q(1-z^2)^{\frac{d-q}{2}-1}}{(1+Kr^2z^2)^{\frac{d+1}{2}} }\, \dd z
\dd r \geq  J(d,q,\gamma, K)$$
with 
$$J(d,q,\gamma, K) = \int^{1/\sqrt{-K}} _{{R_K(u)}}   {r^{q-\gamma-1}}
\int_{0}^{ \frac{R_K(u)}{r}} 
\frac{z^q(1-z^2)^{\frac{d-q}{2}-1}}{(1+Kr^2z^2)^{\frac{d+1}{2}} }\, \dd z
\dd r .$$
We will determine the asymptotics of this term. We first use the substitution $v = r z$, then the substitution $s=\frac{1}{r}$  together with Fubini's theorem to get
\begin{align*}
	J(d,q,\gamma, K) & = \int_{R_K(u)}^{ 1/\sqrt{-K}} {r^{-\gamma-2}} \int_0^{{R_K(u)}}   
\frac{v^q(1- \frac{v^2}{r^2})^{\frac{d-q}{2}-1}}{(1+K v^2)^{\frac{d+1}{2}} }\, \dd v
\dd r \\
& =   \int_{\sqrt{-K}}^{ 1/R_K(u)}{s^{\gamma}} \int_0^{{R_K(u)}}   
\frac{v^q(1- {s^2v^2})^{\frac{d-q}{2}-1}}{(1+K v^2)^{\frac{d+1}{2}} }\, \dd v
\dd s\\
& = \int_0^{{R_K(u)}} v^q (1+K v^2)^{-\frac{d+1}{2}}  \int_{\sqrt{-K}}^{ 1/R_K(u)} {s^{\gamma}} 
{(1- {s^2v^2})^{\frac{d-q}{2}-1}}\, \dd s
\dd v.
\end{align*}
Now we will use Laplace's method for asymptotic approximations of integrals for the latter expression. More precisely, an application of \cite[Theorem 1 on page 58]{wong} gives
$$ \int_{\sqrt{-K}}^{ 1/R_K(u)} {s^{\gamma}} 
{(1- {s^2v^2})^{\frac{d-q}{2}-1}}\, \dd s \sim \Gamma \Big(\frac{\gamma+1}{2}\Big) 2^{\frac{\gamma-1}{2}} v^{-\gamma-1} d^{-\frac{\gamma+1}{2}},  \quad d \to \infty.$$
From this we obtain
\begin{align}\label{asympsub3}
\begin{split}
	& \int_{0}^{1/\sqrt{-K}}    {r^{q-\gamma-1}}
\int_{0}^{1\wedge \frac{R_K(u)}{r}} 
\frac{z^q(1-z^2)^{\frac{d-q}{2}-1}}{(1+Kr^2z^2)^{\frac{d+1}{2}} }\, \dd z
\dd r \\
& \geq \Gamma \Big(\frac{\gamma+1}{2}\Big) 2^{\frac{\gamma-1}{2}}  d^{-\frac{\gamma+1}{2}}  \int_0^{{R_K(u)}} v^{q-\gamma-1} (1+K v^2)^{-\frac{d+1}{2}} \dd v \\
& \sim \Gamma \Big(\frac{\gamma+1}{2}\Big) 2^{\frac{\gamma-1}{2}} \frac{u^{q-\gamma}}{q-\gamma} d^{-\frac{\gamma+1}{2}}  ,
\end{split}
\end{align}
where we made use of the assumption $-K (d) d \to 0$, $d \to \infty$, and the fact that $R_K (u)  \to u$, $K \uparrow 0$, in the last step. Combining \eqref{asympsub1}, \eqref{asympsub2} and \eqref{asympsub3} eventually yields
\begin{align*}
	\liminf_{d \to \infty} \mathbb{P}^{(u)}_{K} (  \mathbf{E} \cap \mathbf{L}\neq \varnothing ) & \geq \lim_{d \to \infty}  \omega_{\gamma+1} \omega_{q-\gamma}  \Big( \frac{1}{2\pi } \Big)^{\frac{\gamma+1}{2}} d^{\frac{\gamma+1}{2}} \frac{q-\gamma}{\omega_{q-\gamma}u^{q-\gamma}} \Gamma \Big(\frac{\gamma+1}{2}\Big) 2^{\frac{\gamma-1}{2}} \frac{u^{q-\gamma}}{q-\gamma} d^{-\frac{\gamma+1}{2}} \\
	& = \omega_{\gamma+1} \Big( \frac{1}{2\pi }\Big)^{\frac{\gamma+1}{2}}  \Gamma \Big(\frac{\gamma+1}{2}\Big) 2^{\frac{\gamma-1}{2}}  =1,
\end{align*}
by definition of $ \omega_{\gamma+1}$. This shows that
$$\lim_{d \to \infty} \mathbb{P}^{(u)}_{K} (  \mathbf{E} \cap \mathbf{L}\neq \varnothing ) = 1,$$
as claimed. 

Now we turn to the supercritical case. That is, we assume that
$$\lim_{d\to \infty}-K(d) d = \infty.$$
By Remark \ref{re:changecurva}, if we denote, as above, by $p_{K,u}$ the probability that $\mathbf{E}$ and $\mathbf{L}$ in $\mathbb{M}^d_{K}$ intersect,  then $p_{K,u}=p_{-1,\sqrt{-K}u}$. Thus, abbreviating $R_{-1} (x) = R(x)$ for $x>0$, by Theorem \ref{intersection probability} we get
$$p_{K,u}=p_{-1,\sqrt{-K}u} 
= \frac{D(d,q,\gamma)\omega_{d-\gamma}}{C_{-1}(d,d-q+\gamma, \sqrt{-K} u)}
\int_{0}^{1}   {r^{q-\gamma-1}}
\int_{0}^{1\wedge \frac{R(\sqrt{-K}u)}{r}} 
\frac{z^q(1-z^2)^{\frac{d-q}{2}-1}}{(1-r^2z^2)^{\frac{d+1}{2}} }\, \dd z
\dd r.$$
Next we use the estimate
\begin{align*}
C_{K}(d,d-q+\gamma, u) & =  (\sqrt{-K}) {\omega_{q-\gamma}} \int_0^u \cosh^{d-q+\gamma} (\sqrt{-K} s)  \sinh^{q-\gamma-1} (\sqrt{-K} s)  ds \\
& \geq  (\sqrt{-K}) {\omega_{q-\gamma}} \int_{u(1-\varepsilon)}^u \cosh^{d-q+\gamma} (\sqrt{-K} s)  \sinh^{q-\gamma-1} (\sqrt{-K} s)  ds\\
& \geq c (\sqrt{-K})^{q-\gamma} \exp\Big( -\frac{u^2(1-\varepsilon)^2}{2} Kd  \Big)
\end{align*}
for some positive and finite constant $c \in (0,\infty)$ and some small $\varepsilon \in (0,1)$, where in the last step we have used the Taylor expansions of both $\cosh(x)$ and $\sinh(x)$ at $x=0$, as well as the assumption $\lim_{d\to \infty}-K(d) d = \infty.$ This together with \eqref{asympsub1} shows that
\begin{equation}\label{asympsup1}
\frac{D(d,q,\gamma)\omega_{d-\gamma}}{C_{-1}(d,d-q+\gamma, \sqrt{-K} u)} = O\Big( d^{\frac{\gamma+1}{2}} (\sqrt{-K})^{\gamma-q} \exp\big( \frac{u^2(1-\varepsilon)^2}{2} Kd  \big)
\Big) , \quad d \to \infty .
\end{equation}
Now we estimate the integral
$$\int_{0}^{1}   {r^{q-\gamma-1}}
\int_{0}^{1\wedge \frac{R(\sqrt{-K}u)}{r}} 
\frac{z^q(1-z^2)^{\frac{d-q}{2}-1}}{(1-r^2z^2)^{\frac{d+1}{2}} }\, \dd z
\dd r.$$
We split the integral into two parts:
$$\int_{0}^{1}   {r^{q-\gamma-1}}
\int_{0}^{1\wedge \frac{R(\sqrt{-K}u)}{r}} 
\frac{z^q(1-z^2)^{\frac{d-q}{2}-1}}{(1-r^2z^2)^{\frac{d+1}{2}} }\, \dd z
\dd r = J_1 + J_2,$$
with 
$$J_1 = \int_{0}^{{R(\sqrt{-K}u)}}   {r^{q-\gamma-1}}
\int_{0}^{1} 
\frac{z^q(1-z^2)^{\frac{d-q}{2}-1}}{(1-r^2z^2)^{\frac{d+1}{2}} }\, \dd z
\dd r$$
and 
$$J_2 = \int^1_{{R(\sqrt{-K}u)}}   {r^{q-\gamma-1}}
\int_{0}^{ \frac{R(\sqrt{-K}u)}{r}} 
\frac{z^q(1-z^2)^{\frac{d-q}{2}-1}}{(1-r^2z^2)^{\frac{d+1}{2}} }\, \dd z
\dd r .$$
Applying \cite[Theorem 1 on page 58]{wong} in the same manner as before shows that for large enough $d$:
$$\int_{0}^{1} 
\frac{z^q(1-z^2)^{\frac{d-q}{2}-1}}{(1-r^2z^2)^{\frac{d+1}{2}} }\, \dd z \leq c (1-r^2) d^{-\frac{q+1}{2}}$$
for some constant $c \in (0,\infty)$. This yields that
\begin{equation}\label{asympsup21}
J_1= O \Big( d^{-\frac{q+1}{2}} (\sqrt{-K})^{q-\gamma} \Big) , \quad d \to \infty .
\end{equation}
On the other hand, for the term $J_2$ with similar arguments we obtain
\begin{equation}\label{asympsup22}
J_2= O \big( d^{-\frac{q+1}{2}}  \big) , \quad d \to \infty .
\end{equation}
Combining \eqref{asympsup21} and \eqref{asympsup22} yields
\begin{equation}\label{asympsup2}
\int_{0}^{1}   {r^{q-\gamma-1}}
\int_{0}^{1\wedge \frac{R(\sqrt{-K}u)}{r}} 
\frac{z^q(1-z^2)^{\frac{d-q}{2}-1}}{(1-r^2z^2)^{\frac{d+1}{2}} }\, \dd z
\dd r = O \big( d^{-\frac{q+1}{2}}  \big) , \quad d \to \infty .
\end{equation}
Thus, using \eqref{asympsup1} and \eqref{asympsup2} we obtain
\begin{align*}
	\limsup_{d \to \infty} \mathbb{P}^{(u)}_{K} (  \mathbf{E} \cap \mathbf{L}\neq \varnothing ) & \leq c \lim_{d \to \infty} d^{\frac{\gamma+1}{2}} (\sqrt{-K})^{\gamma-q} \exp\big( \frac{u^2(1-\varepsilon)^2}{2} Kd  \big) d^{-\frac{q+1}{2}} \\
	& = \lim_{d \to \infty} (-Kd)^{-\frac{q-\gamma}{2}}  \exp\big( \frac{u^2(1-\varepsilon)^2}{2} Kd  \big) = 0,
\end{align*}
where in the last step we also used the assumption $\lim_{d\to \infty}-K(d) d = \infty.$

It remains to consider the critical case. That is, we assume that 
$$ -K =  -K (d) \sim \frac{\kappa}{d} \quad \textnormal{ for some constant } \kappa \in (0,\infty).$$
In this case we will show that
$$ \lim_{d\to \infty}\mathbb{P} ( \mathbf{E} \cap \mathbf{L}\neq \varnothing ) =  \rho(u,q,\gamma, \kappa) ,$$
where the constant $\rho(u,q,\gamma, \kappa)$ is as in the statement of Theorem \ref{phase transition}. Recall from above that
$$p_{K,u}=p_{-1,\sqrt{-K}u} 
= \frac{D(d,q,\gamma)\omega_{d-\gamma}}{C_{-1}(d,d-q+\gamma, \sqrt{-K} u)}
\int_{0}^{1}   {r^{q-\gamma-1}}
\int_{0}^{1\wedge \frac{R(\sqrt{-K}u)}{r}} 
\frac{z^q(1-z^2)^{\frac{d-q}{2}-1}}{(1-r^2z^2)^{\frac{d+1}{2}} }\, \dd z
\dd r.$$
Now we calculate exact asymptotic expressions, for $d \to \infty$, step by step. 

\paragraph{\textbf{Step 1:}} Recall that we have equation \eqref{asympsub1}:
$$
D(d,q,\gamma)\omega_{d-\gamma} \sim  \omega_{\gamma+1} \omega_{q-\gamma}  \Big( \frac{1}{2\pi } \Big)^{\frac{\gamma+1}{2}} d^{\frac{\gamma+1}{2}}, \quad d \to \infty .
$$ 
\paragraph{\textbf{Step 2:}} By an observation already made it is easy to see that in the current setting we have 
 \begin{align}\label{asymp2}
 \begin{split}
C_{-1}(d,d-q+\gamma, \sqrt{-K(d)} u) &  \sim {\omega_{q-\gamma}} \int_0^u e^{\frac{\kappa}{2}s^2} s^{q-\gamma-1} \,\dd s (\sqrt{-K(d)})^{q-\gamma} \\
& \sim  {\omega_{q-\gamma}} \int_0^u e^{\frac{\kappa}{2}s^2} s^{q-\gamma-1} \,\dd s \kappa^{\frac{q-\gamma}{2}} d^{-\frac{q-\gamma}{2}} , \  \quad d \to \infty .
 \end{split}
\end{align}
\paragraph{\textbf{Step 3:}} It remains to calculate the asymptotics of the expression
$$\int_{0}^{1}   {r^{q-\gamma-1}}
\int_{0}^{1\wedge \frac{R(\sqrt{-K}u)}{r}} 
\frac{z^q(1-z^2)^{\frac{d-q}{2}-1}}{(1-r^2z^2)^{\frac{d+1}{2}} }\, \dd z
\dd r.$$
As before we split the integral into two parts:
$$\int_{0}^{1}   {r^{q-\gamma-1}}
\int_{0}^{1\wedge \frac{R(\sqrt{-K}u)}{r}} 
\frac{z^q(1-z^2)^{\frac{d-q}{2}-1}}{(1-r^2z^2)^{\frac{d+1}{2}} }\, \dd z
\dd r = J_1 + J_2,$$
with 
$$J_1 = \int_{0}^{{R(\sqrt{-K}u)}}   {r^{q-\gamma-1}}
\int_{0}^{1} 
\frac{z^q(1-z^2)^{\frac{d-q}{2}-1}}{(1-r^2z^2)^{\frac{d+1}{2}} }\, \dd z
\dd r$$
and 
$$J_2 = \int^1_{{R(\sqrt{-K}u)}}   {r^{q-\gamma-1}}
\int_{0}^{ \frac{R(\sqrt{-K}u)}{r}} 
\frac{z^q(1-z^2)^{\frac{d-q}{2}-1}}{(1-r^2z^2)^{\frac{d+1}{2}} }\, \dd z
\dd r .$$
We will show that the leading term in the asymptotics is given by $J_2$. We calculate the asymptotics of $J_1$ first. Using the substitution $s = \frac{1}{R(\sqrt{-K}u)} r$ in the first step, then the substitution $v = \frac{1}{R(\sqrt{-K}u)} z$ in the second step, and the fact that $R(\sqrt{-K}u) \sim u\sqrt{\kappa/d}$ we get
\begin{align*}
	J_1 & = R(\sqrt{-K}u)^{1+q-\gamma-1} \int_{0}^{{1}}   {s^{q-\gamma-1}}
\int_{0}^{1} 
\frac{z^q(1-z^2)^{\frac{d-q}{2}-1}}{(1-R(\sqrt{-K}u)^2s^2z^2)^{\frac{d+1}{2}} }\, \dd z
\dd s \\
& = R(\sqrt{-K}u)^{1+2q-\gamma} \int_{0}^{{1}}   {s^{q-\gamma-1}}
\int_{0}^{\frac{1}{R(\sqrt{-K}u)}} 
\frac{v^q(1-R(\sqrt{-K}u)^2v^2)^{\frac{d-q}{2}-1}}{(1-R(\sqrt{-K}u)^4s^2v^2)^{\frac{d+1}{2}} }\, \dd v
\dd s  \\
&\sim (u\sqrt{\kappa})^{1+2q-\gamma}  d^{-\frac{1+2q-\gamma}{2}} \int_{0}^{{1}}   {s^{q-\gamma-1}}
\int_{0}^{\frac{1}{R(\sqrt{-K}u)}} 
\frac{v^q(1-R(\sqrt{-K}u)^2v^2)^{\frac{d-q}{2}-1}}{(1-R(\sqrt{-K}u)^4s^2v^2)^{\frac{d+1}{2}} }\, \dd v
\dd s .
\end{align*}
Observe that
$$\lim_{d \to \infty}  = \frac{ (1-R(\sqrt{-K}u)^2v^2)^{\frac{d-q}{2}-1}}{(1-R(\sqrt{-K}u)^4s^2v^2)^{\frac{d+1}{2}} } = \exp\Big( - \frac12 u^2 \kappa v^2 \Big) $$
and that
 $$\int_{0}^{{1}}   {s^{q-\gamma-1}} \dd s
\int_{0}^{{\infty}} 
 {v^q}  \exp\Big( - \frac12 u^2 \kappa v^2 \Big)  \, \dd v < \infty.
$$
Thus, we obtain
\begin{equation}\label{asymp31}
J_1 = O( d^{-\frac{1+2q-\gamma}{2}}  ) = o( d^{-\frac{1+q}{2}} ) , \quad d\to \infty .
\end{equation}
It remains to calculate the asymptotics of $J_2$. Here we only use the substitution $v = \frac{1}{R(\sqrt{-K}u)} z$ and then the fact that $R(\sqrt{-K}u) \sim u\sqrt{\kappa/d} $:
\begin{align*}
	J_2 & =  R(\sqrt{-K}u)^{1+q} \int^1_{{R(\sqrt{-K}u)}}   {r^{q-\gamma-1}}
\int_{0}^{ \frac{1}{r}} 
\frac{v^q(1-R(\sqrt{-K}u)^2v^2)^{\frac{d-q}{2}-1}}{(1-R(\sqrt{-K}u)^2 r^2v^2)^{\frac{d+1}{2}} }\, \dd v
\dd r \\
& \sim  u^{q+1}  \kappa^{\frac{q+1}{2}} d^{-\frac{q+1}{2}}  \int^1_{{R(\sqrt{-K}u)}}   {r^{q-\gamma-1}}
\int_{0}^{ \frac{1}{r}} 
\frac{v^q(1-R(\sqrt{-K}u)^2v^2)^{\frac{d-q}{2}-1}}{(1-R(\sqrt{-K}u)^2 r^2v^2)^{\frac{d+1}{2}} }\, \dd v
\dd r.
\end{align*}
Now, as before, observing that
$$\lim_{d \to \infty}  = \frac{ (1-R(\sqrt{-K}u)^2v^2)^{\frac{d-q}{2}-1}}{(1-R(\sqrt{-K}u)^2r^2v^2)^{\frac{d+1}{2}} } = \exp\Big( - \frac12 u^2 \kappa v^2 \Big) \exp\Big(  \frac12 u^2 \kappa r^2v^2 \Big) $$
and that
 $$\int_{0}^{{1}}   {r^{q-\gamma-1}} 
\int_{0}^{{\frac{1}{r}}} 
 {v^q}  e^{ - \frac12 u^2 \kappa v^2 } e^{  \frac12 u^2 \kappa r^2v^2 } \, \dd v \dd r< \infty
$$
we obtain  
\begin{equation}\label{asymp32}
	J_2  \sim  u^{q+1}  \kappa^{\frac{q+1}{2}} d^{-\frac{q+1}{2}}  \int^1_{{0}}   {r^{q-\gamma-1}}
\int_{0}^{ \frac{1}{r}} 
{v^q}  e^{ - \frac12 u^2 \kappa v^2 } e^{  \frac12 u^2 \kappa r^2v^2 } \, \dd v
\dd r.
\end{equation}
Combining \eqref{asymp31} with \eqref{asymp32}  arrive at 
\begin{align}\label{asymp3}
	\begin{split}
	& \int_{0}^{1}   {r^{q-\gamma-1}}
\int_{0}^{1\wedge \frac{R(\sqrt{-K}u)}{r}} 
\frac{z^q(1-z^2)^{\frac{d-q}{2}-1}}{(1-r^2z^2)^{\frac{d+1}{2}} }\, \dd z
\dd r \\
&  \sim  u^{q+1}  \kappa^{\frac{q+1}{2}} d^{-\frac{q+1}{2}}  \int^1_{{0}}   {r^{q-\gamma-1}}
\int_{0}^{ \frac{1}{r}} 
{v^q}  e^{ - \frac12 u^2 \kappa v^2 } e^{  \frac12 u^2 \kappa r^2v^2 } \, \dd v
\dd r.
\end{split}
\end{align}
From \eqref{asympsub1}, \eqref{asymp2} and \eqref{asymp3} we deduce that
\begin{align*}
	& \lim_{d\to \infty}\mathbb{P} ( \mathbf{E} \cap \mathbf{L}\neq \varnothing ) \\
	& = \lim_{d\to \infty} \frac{D(d,q,\gamma)\omega_{d-\gamma}}{C_{-1}(d,d-q+\gamma, \sqrt{-K(d)} u)}
\int_{0}^{1}   {r^{q-\gamma-1}}
\int_{0}^{1\wedge \frac{R(\sqrt{-K(d)}u)}{r}} 
\frac{z^q(1-z^2)^{\frac{d-q}{2}-1}}{(1-r^2z^2)^{\frac{d+1}{2}} }\, \dd z
\dd r\\
& = \lim_{d\to \infty}  \omega_{\gamma+1} \omega_{q-\gamma}  \Big( \frac{1}{2\pi } \Big)^{\frac{\gamma+1}{2}} d^{\frac{\gamma+1}{2}}  \frac{1}{\omega_{q-\gamma} \kappa^{\frac{q-\gamma}{2}}}  d^{\frac{q-\gamma}{2}}  u^{q+1}  \kappa^{\frac{q+1}{2}} \Big( \int_0^u e^{\frac{\kappa}{2}s^2} s^{q-\gamma-1} \,\dd s \Big)^{-1} \\
& \hspace{4cm} \times d^{-\frac{q+1}{2}}  \int^1_{{0}}   {r^{q-\gamma-1}} 
\int_{0}^{ \frac{1}{r}} 
{v^q}  e^{ - \frac12 u^2 \kappa v^2 } e^{  \frac12 u^2 \kappa r^2v^2 } \, \dd v
\dd r \\
& =  \omega_{\gamma+1}   (2\pi )^{-\frac{\gamma+1}{2}}  u^{q+1} \kappa^{\frac{\gamma+1}{2}} \Big( \int_0^u e^{\frac{\kappa}{2}s^2} s^{q-\gamma-1} \,\dd s \Big)^{-1} \\
& \hspace{4cm}  \times \int^1_{{0}}   {r^{q-\gamma-1}} 
\int_{0}^{ \frac{1}{r}} 
{v^q}  \exp \Big( - \frac12 u^2 \kappa v^2 (1-r^2) \Big)  \, \dd v
\dd r\\
& = \rho(u,q,\gamma, \kappa).
\end{align*}
 \enlargethispage{\baselineskip}
This finishes the proof.
\end{proof}

\bibliographystyle{plain}
\bibliography{lit}

\end{document}